\documentclass{amsart}

\usepackage[numbers]{natbib}
\usepackage{amssymb}
\usepackage{amsthm}
\usepackage{amsmath}
\usepackage{amsfonts}
\usepackage{mathtools}
\usepackage[shortlabels]{enumitem}
\newtheorem{theorem}{Theorem}
\newtheorem{corollary}{Corollary}
\newtheorem{lemma}{Lemma}

\newtheorem{definition}{Definition}
\usepackage{multicol}
\usepackage{tabularx}

\usepackage[pdftitle={Improved linear programming methods for checking avoiding sure loss},%
pdfauthor={Nawapon Nakharutai, Matthias C. M. Troffaes, Camila C. S. Caiado},%
pdfkeywords={avoiding sure loss; linear programming; primal and dual problems; benchmarking; simplex method; affine scaling method; primal-dual method; algorithm}]{hyperref}
\usepackage[nosort]{cleveref}
\usepackage{doi}
\usepackage{tikz}
\usetikzlibrary{shapes.geometric, arrows}
\usetikzlibrary{er,positioning}
\usetikzlibrary{shapes}
\usepackage{subcaption}
\usepackage[font=small,skip=1pt]{caption}
\usepackage{algorithm}
\usepackage{algorithmic}

\newlength\myindent
\makeatletter
\newenvironment{breakablealgorithm}
  {
   \begin{center}
     \refstepcounter{algorithm}
     \hrule height.8pt depth0pt \kern2pt
     \renewcommand{\caption}[2][\relax]{
       {\raggedright\textbf{\ALG@name~\thealgorithm} ##2\par}%
       \ifx\relax##1\relax 
         \addcontentsline{loa}{algorithm}{\protect\numberline{\thealgorithm}##2}%
       \else 
         \addcontentsline{loa}{algorithm}{\protect\numberline{\thealgorithm}##1}%
       \fi
       \kern2pt\hrule\kern2pt
     }
  }{
     \kern2pt\hrule\relax
   \end{center}
  }
\makeatother
\usepackage{multirow}

\DeclareMathOperator{\dom}{dom}

\DeclareMathOperator{\ext}{ext}

\newcommand{\linprogref}[1]{\textup{(#1)}}

\title[Improved methods for checking avoiding sure loss]{Improved linear programming methods for checking avoiding sure loss}

\author{Nawapon Nakharutai}
\address{Durham University, Department of Mathematical Sciences, UK}
\email{nawapon.nakharutai@durham.ac.uk}

\author{Matthias C. M. Troffaes}
\address{Durham University, Department of Mathematical Sciences, UK}
\email{matthias.troffaes@durham.ac.uk}

\author{Camila C. S. Caiado} 
\address{Durham University, Department of Mathematical Sciences, UK}
\email{c.c.d.s.caiado@durham.ac.uk}


\begin{document}

\begin{abstract}
We review the simplex method and two interior-point methods (the affine scaling and the primal-dual) for solving linear programming problems for checking avoiding sure loss, and propose novel improvements. We exploit the structure of these problems to reduce their size. We also present an extra stopping criterion, and direct ways to calculate feasible starting points in almost all cases. 
For benchmarking, we present algorithms for generating random sets of desirable gambles that either avoid or do not avoid sure loss. We test our improvements on these linear programming methods by measuring the computational time on these generated sets. We assess the relative performance of the three methods as a function of the number of desirable gambles and the number of outcomes.
Overall, the affine scaling and primal-dual methods benefit from the improvements, and they both outperform the simplex method in most scenarios. We conclude that the simplex method is not a good choice for checking avoiding sure loss. If problems are small, then there is no tangible difference in performance between all methods. For large problems, our improved primal-dual method performs at least three times faster than any of the other methods.
\end{abstract}

\maketitle

\section{Introduction}
In statistical modelling, we often face issues such as limited structural information about dependencies, lack of data, limited expert opinion, or even contradicting information from different experts.
Various authors \cite{1975:williams:condprev,2007:williams:condprev,1991:walley,2014:troffaes:decooman::lower:previsions} have argued that these issues can be handled by modelling our beliefs using \emph{sets of desirable gambles}. A gamble represents a reward (e.g. monetary) that depends on an uncertain outcome.
We can model our beliefs about this outcome by stating a collection of gambles that we are willing to accept. Such set is called a set of desirable gambles. Through duality, every set of desirable gambles is mathematically equivalent to a set of probability distributions.

If there are no combinations of desirable gambles that result in a certain loss, then we say that our set of desirable gambles \emph{avoids sure loss} \cite{1975:williams:condprev,2007:williams:condprev}. To verify whether a set of desirable gambles avoids sure loss, we can solve a linear programming problem \cite[p.~151]{1991:walley}.

Linear programs for checking avoiding sure loss have been studied for instance in \cite{2004:walley:condlowprev,Quaeghebeur2014}. However, these studies focus on forming linear programs and do not mention which algorithms we should use. In the early '90s, Walley \citep[p.~551]{1991:walley} mentioned that  Karmarkar's method can be considered for solving large linear programs. However, nowadays Karmarkar's method is considered obsolete  in favour of other interior point methods such as affine scaling and primal-dual methods \citep{Anstreicher:2009}.

The simplex method is one of the oldest and simplest methods, and the affine scaling method is an improved version of Karmarkar's method, whilst the primal-dual method is currently considered one of the best general purpose methods.
In previous work, we presented an initial comparative study of these three methods for checking avoiding sure loss \cite{2017:Nakharutai:Troffaes:Caiado}.
In that study, we slightly reduced the size of the problems and proposed two improvements: an extra stopping criterion to detect unboundedness more quickly, and a direct way to calculate feasible starting points. There, we also quantified the impact of these improvements \citep[Fig.~1]{2017:Nakharutai:Troffaes:Caiado}, for the primal-dual method.

In this paper, our main goal is to elaborate on the improvements in \cite{2017:Nakharutai:Troffaes:Caiado}, and to further develop efficient algorithms for checking avoiding sure loss. In particular, we study also the dual of each linear program, and we generalise the process to find feasible starting points. We also discuss in detail the advantages and disadvantages of each method for checking avoiding sure loss. For benchmarking, we provide a variety of algorithms for generating sets of desirable gambles. In a simulation study, we generate random sets of desirable gambles and assess the impact of our improvements. In addition, we provide proofs for all relevant results, including some that were stated without proof in \cite{2017:Nakharutai:Troffaes:Caiado}.

The paper is organised as follows. \Cref{sec2} gives a brief outline of avoiding sure loss and coherence. \Cref{sec3} studies several linear programming problems for checking avoiding sure loss, and we slightly reduce the size of these linear programming problems. \Cref{sec4} reviews the simplex, the affine scaling and the primal-dual algorithms, and studies how we can improve these algorithms to check avoiding sure loss. \Cref{sec7,sec7next} present some algorithms for generating random sets of desirable gambles. \Cref{sec8} compares the efficiency of our improved methods.
\Cref{sec10} concludes the paper.

\section{Desirable gambles and lower previsions}\label{sec2}

In this section, we explain desirable gambles, lower previsions, avoiding sure loss, coherence, and natural extension \cite{1991:walley}. We also introduce the notation used throughout.

\subsection{Avoiding sure loss}

Let $\Omega$ be a finite set of uncertain outcomes. A \emph{gamble} is a bounded real-valued function on $\Omega$. Let $\mathcal{L}(\Omega)$ denote the set of all gambles on $\Omega$. Let $\mathcal{D} $ be a finite set of gambles that a subject decides to accept; we call $\mathcal{D}$ the subject's \emph{set of desirable gambles}.
The desirability axioms essentially state that a non-negative combination of desirable gambles should not produce a sure loss \citep[\S 2.3.3]{1991:walley}.
In that case, we say that $\mathcal{D}$ avoids sure loss.

\begin{definition}\citep[\S 3.7.1]{1991:walley}\label{def:1}
  A set $\mathcal{D}\subseteq\mathcal{L}(\Omega)$ is said to \emph{avoid sure loss} if for
  all $ n \in \mathbb{N}$, all  $\lambda_{1}, \dots,\lambda_{n} \geq 0$, and all $f_{1}, \dots,f_{n} \in \mathcal{D}$,
\begin{equation}\label{eq:2.1}
\max_{\omega \in \Omega}\left(\sum_{i=1}^{n} \lambda_{i}f_{i}(\omega)\right)  \geq 0.
\end{equation}
\end{definition}

We can also model uncertainty via acceptable buying (or selling) prices for gambles. A \emph{lower prevision} $\underline{P}$ is a real-valued function defined on some subset of $\mathcal{L}(\Omega)$. We denote the domain of $\underline{P}$ by $\dom\underline{P}$. Given a gamble $f\in\dom\underline{P}$, we interpret $\underline{P}(f)$ as a subject's supremum buying price for $f$.

\begin{definition}\citep[\S 2.4.2]{1991:walley}\label{def:47}
  A lower prevision $\underline{P}$ is said to \emph{avoid sure loss} if for
  all $ n \in \mathbb{N}$, all  $\lambda_{1}, \dots,\lambda_{n} \geq 0$, and all $f_{1}, \dots,f_{n} \in \dom\underline{P}$,
\begin{equation}\label{eq:3.1}
\max_{\omega\in \Omega} \left( \sum_{i=1}^{n} \lambda_{i}\left[f_{i}(\omega)-\underline{P}(f_{i})\right] \right) \geq 0. 
\end{equation}  
\end{definition}

Any lower prevision $\underline{P}$ induces a \emph{conjugate upper prevision} $\overline{P}$ on $-\dom \underline{P}\coloneqq\{-f\colon f\in \dom \underline{P}\}$, defined by $\overline{P}(f)\coloneqq -\underline{P}(-f)$ for all $f\in -\dom\underline{P}$ \citep[\S 2.3.5]{1991:walley}. $\overline{P}(f)$ represents a subject's infimum selling price for $f$.
$\underline{P}$ is said to be \emph{self-conjugate} when $\dom\underline{P}=-\dom\underline{P}$ and $\underline{P}(f) = \overline{P}(f)$ for all $f \in \dom P$. We simply call a self-conjugate lower prevision $\underline{P}$ a \emph{prevision} and write it as $P$
\cite[p.~41]{2014:troffaes:decooman::lower:previsions}.

\subsection{Coherence}\label{coherent}

Coherence is another rationality condition for lower previsions and is stronger than avoiding sure loss. Coherence requires that the subject's supremum buying prices for gambles cannot be increased by considering any finite non-negative linear combination of other desirable gambles \cite[\S 2.5.2]{1991:walley}. In \cref{sec7}, we will use coherent lower previsions to generate sets of desirable gambles that avoids sure loss.

\begin{definition}\cite[\S 2.5.4]{1991:walley}
A lower prevision $\underline{P}$ is said to be \emph{coherent} if 
for all $n \in \mathbb{N}$, all $\lambda_{0},\dots,\lambda_{n} \geq 0$ and all $f_0, \dots, f_n \in\dom \underline{P}$, 
\begin{equation}\label{eq8.2}
\sup_{\omega \in \Omega}\left(\sum_{i=1}^{n}\lambda_i [f_i(\omega)- \underline{P}(f_i)] -\lambda_0[f_0(\omega) - \underline{P}(f_0)] \right) \geq 0.
\end{equation}
\end{definition}

Next, we give some examples of coherent lower previsions. The lower prevision given by $\underline{P}(f) \coloneqq \inf f$ for all $f \in \mathcal{L}(\Omega)$ is coherent, and is called the \emph{vacuous} lower prevision \cite[\S 2.3.7]{1991:walley}.
Previsions that avoid sure loss are also coherent:

\begin{theorem}
 \citep[p.87]{1991:walley}
A prevision $P$ is coherent if and only if it avoids sure loss (as a lower prevision).
\end{theorem}

The expectation of $f$ associated with the probability mass function $p$ is given by
\begin{equation}\label{eq:8.3}
E_p(f) \coloneqq \sum_{\omega \in \Omega}p(\omega)f(\omega).
\end{equation}
An expectation operator is coherent as well.

Generating probability mass functions is easy (see \cref{alg:Exp} further), and we can use them
to generate other coherent lower previsions via lower envelopes and convex combinations:

\begin{definition}\cite[p.~60]{2014:troffaes:decooman::lower:previsions}\label{def:4}
Let $\Gamma$ be a non-empty collection of lower previsions defined on a common domain $\mathcal{K}$. A lower prevision $\underline{Q}$ is called the \emph{lower envelope} of $\Gamma$ if 
\begin{equation}
\underline{Q}(f) = \inf_{\underline{P} \in \Gamma}\underline{P}(f) \text{ for all } f \in \mathcal{K}.
\end{equation}
\end{definition}
\begin{theorem}\cite[p.~61]{2014:troffaes:decooman::lower:previsions}\label{thm:2}
If all lower previsions in  $\Gamma$  are coherent, then the lower envelope of $\Gamma$ is also coherent.
\end{theorem}
We define the unit simplex as the set of all probability mass functions:
\begin{equation}
\Delta(\Omega) \coloneqq \left\{ p \in \mathbb{R}^\Omega\colon p \geq 0\text{\ and }\sum_{\omega \in \Omega} p(\omega) = 1 \right\}.
\end{equation}
Its extreme points are the $\{0,1\}$-valued probability mass functions \cite[\S 3.2.6]{1991:walley}. 
The \emph{credal set} of a lower prevision $\underline{P}$ is defined by
\begin{equation}
\mathcal{M}_{\underline{P}} := \{p \in \Delta(\Omega):\forall f \in \dom\underline{P},\ E_p(f) \geq\underline{P}(f) \}.
\end{equation}
$\mathcal{M}_{\underline{P}}$  completely determines $\underline{P}$ if $\underline{P}$ is coherent and there is a one-to-one correspondence between coherent lower previsions on $\mathcal{L}(\Omega)$ and closed convex subsets of $\Delta(\Omega)$ \cite[p.~79]{2014:troffaes:decooman::lower:previsions}. Moreover, it suffices to consider the set of extreme points $\ext\mathcal{M}_{\underline{P}}$ of $\mathcal{M}_{\underline{P}}$ \cite[p.~145]{1991:walley}:

\begin{theorem}(adapted from \cite[p.~146]{1991:walley})\label{thm:73}
Let $\underline{P}$ be a coherent lower prevision. Then for every $f\in\dom\underline{P}$, there is a $p\in\ext\mathcal{M}_{\underline{P}}$ such that $\underline{P}(f) = E_p(f)$.
\end{theorem}

When $\ext\mathcal{M}_{\underline{P}}$ is finite, then $\underline{P}$ is called \emph{polyhedral}. 
We can construct a polyhedral lower prevision as follows.
Let $M$ be a finite set of probability mass functions on $\Omega$. A polyhedral lower prevision is then given by  \cite[\S 9.2.1]{2014:itip}
\begin{equation}\label{eq:8.8}
\underline{E}_{M}(f) \coloneqq\min_{p \in M}E_p(f).
\end{equation}

\begin{theorem}\citep[p.~79]{1991:walley}\label{thm:2.5}
  Let $\underline{P}_1$ and $\underline{P}_2$ be lower previsions on the same domain. Let $0\le \delta \le 1$.
  If $\underline{P}_1$ and $\underline{P}_2$ are coherent, then so is $(1-\delta)\underline{P}_1+\delta\underline{P}_2$.
\end{theorem}

Let $P_0$ be a coherent prevision on $\mathcal{L}(\Omega)$ and let $0 \le \delta \le 1$.
The lower prevision defined
on all $f\in\mathcal{L}(\Omega)$
by
\begin{equation}\label{eq:8.2}
\underline{P}(f) \coloneqq (1-\delta)P_{0}(f) + \delta \inf f
\end{equation}
is called a \emph{linear-vacuous mixture} \citep[\S 2.9.2]{1991:walley},
and is coherent by \cref{thm:2.5}.


\subsection{Natural extension}


The natural extension of a set of desirable gambles $\mathcal{D}$ is defined as the smallest set of gambles which includes all finite non-negative combinations of gambles in $\mathcal{D}$ and all non-negative gambles \cite[\S~3.7]{2014:troffaes:decooman::lower:previsions}:

\begin{definition}\cite[p.~32]{2014:troffaes:decooman::lower:previsions}
The \emph{natural extension} of a set $\mathcal{D}\subseteq\mathcal{L}(\Omega)$ is:
\begin{equation}\label{eq:5.1}
\mathcal{E}_{\mathcal{D}}\coloneqq\left\lbrace g_{0} + \sum_{i=1}^{n} \lambda_{i}g_{i} \colon g_{0} \geq 0,\,n \in \mathbb{N},\, g_1,\dots,g_n \in \mathcal{D},\,\lambda_1,\dots,\lambda_n \geq 0 \right\rbrace.
\end{equation}
\end{definition}

From this natural extension, we can derive a supremum buying price for any gamble $f$:
\begin{definition}
For any set $\mathcal{D}\subseteq\mathcal{L}(\Omega)$ and $f\in\mathcal{L}(\Omega)$, we define:
\begin{align}\label{eq:6.11}
\begin{split}
\underline{E}_{\mathcal{D}}(f) & \coloneqq \sup \left\lbrace\alpha \in \mathbb{R}\colon f-\alpha \in \mathcal{E}_{\mathcal{D}}\right\rbrace\\
& =  \sup \left\lbrace\alpha \in \mathbb{R}\colon f -\alpha \geq \sum_{i=1}^{n} \lambda_{i}f_{i}, n \in \mathbb{N}, f_{i} \in \mathcal{D}, \lambda_{i} \geq 0\right\rbrace.
\end{split}
\end{align}
\end{definition}
Note that we can derive a set of desirable gambles from $\underline{P}$ as follows \citep[p.~42]{2014:troffaes:decooman::lower:previsions}:
\begin{equation}\label{eq:5.10}
\mathcal{D}_{\underline{P}} \coloneqq \left\{ g-\mu\colon g \in \dom\underline{P} \text{ and } \mu < \underline{P}(g) \right\}.
\end{equation} 
Putting this all together, we can define the natural extension of $\underline{P}$:
\begin{definition} \citep[\S 3.1.1]{1991:walley}
Let $\underline{P}$ be a lower prevision. The natural extension of $\underline{P}$ is a lower prevision defined on all $f \in \mathcal{L}(\Omega)$ and given by:
\begin{multline}
\underline{E}_{\underline{P}}(f) \coloneqq \underline{E}_{\mathcal{D}_{\underline{P}}}(f)
\\
=  \sup \left\lbrace\alpha \in \mathbb{R}\colon f -\alpha \geq \sum_{i=1}^{n} \lambda_{i}(f_{i}-\underline{P}(f_{i})) , n\in\mathbb{N},\,f_{i} \in  \dom\underline{P},\,\lambda_{i} \geq 0\right\rbrace. 
\end{multline}
\end{definition} 

\section{Linear programming}\label{sec3}
In this section, we briefly review linear programming, and we study several linear programming problems for checking avoiding sure loss.

\subsection{Linear programming problems}
Any problem of minimising or maximising a linear function, called the objective function, subject to linear constraints, is called a linear program, and can always be written in the form of \cref{eq:p}, with dual given in \cref{eq:d}:
\begin{align}
  \tag{P}\label{eq:p}
  \min c^\intercal x
  &\text{ subject to } Ax = b,\ x \geq 0 \\
  \tag{D}\label{eq:d}
  \max b^\intercal y
  &\text{ subject to } A^\intercal y + t = c,\ t \geq 0, \ y \text{ free}
\end{align}
where $A \in \mathbb{R}^{m \times n}$ has rank $m$, $m \leq n$,
and $x$, $c$, $t$, $b$, and $y$ are vectors with dimensions as expected
(so $x$, $c$, $t\in\mathbb{R}^n$ and $b$, $y\in \mathbb{R}^m$).
We call \cref{eq:p} the \emph{primal problem} and \cref{eq:d} the \emph{dual problem}. They have the same solution \citep[p.~59]{1993:Fang:Puthenpura}, so we can solve either one of them. 

A solution is called \emph{feasible} if it satisfies all constraints. The primal problem is called \emph{unbounded} if for all $\lambda \in \mathbb{R}$, there is a feasible solution $x$ such that $c^\intercal x\leq \lambda$. A feasible solution that achieves the optimal value of the objective function is called an \emph{optimal} solution. A \emph{basis} is a collection of $m$ of the $n$ variables that correspond to $m$ linearly independent columns of $A$.
A variable in the basis is called a \emph{basic} variable; otherwise it is called a \emph{non-basic} variable. A \emph{basic} feasible solution is a feasible solution such that all of the $n-m$ non-basic variables are zero. The basic feasible solutions are precisely the extreme points of the feasible region.

A linear programming problem is called \emph{degenerate} if it has basic feasible solutions with $n-m+1$ or more zero elements \citep[p.~22]{1993:Fang:Puthenpura}. If all basic feasible solutions are zero (this happens when $b=0$), we say that the problem is \emph{fully} degenerate. In this case, the feasible region is a cone, and therefore has only one extreme point, namely, the origin. As we will see later, one way to check avoiding sure loss is to solve a fully degenerate problem.
The following lemma is useful for finding the optimal value of a fully degenerate problem.

\begin{lemma}\label{lem:1} (A generalised version of \citep[p.~42, exercise~3.4]{2001:Vanderbei}.)
The linear programming problem
\begin{align}\label{eqlem1:1}
   \min \quad c^\intercal x \\ \label{eqlem1:2}
    \text{subject to }
      \quad Ax & \geq 0 
\end{align}
either has optimal value equal to zero, or is unbounded below.
\end{lemma}
\Cref{lem:1} shows that, for fully degenerate problems, if there is a feasible solution $x$ such that $c^\intercal x <0 $, then the problem is unbounded. Therefore, in our algorithms, we can stop early as soon as we find a negative value.

\subsection{Linear programming for checking avoiding sure loss}

We now present some linear programming problems for checking avoiding sure loss. The problems \linprogref{P1} and \linprogref{D1} in \cref{thm:1} are similar to the linear programming problems discussed for lower previsions in \citet[p.~175]{1991:walley}.

\begin{theorem}\label{thm:1}
The set $\mathcal{D}=\{f_1,\dots,f_n\}$ avoids sure loss if and only if the optimal value of \linprogref{P1} is zero:
\begin{align}
   \label{thm1:1}\tag{P1a}
   \linprogref{P1} &&
   \min \quad & \alpha \\
   \label{thm1:2}\tag{P1b}
   && \text{subject to}\quad
   &\forall \omega \in \Omega\colon
    \sum_{i=1}^{n} f_{i}(\omega)\lambda_{i} -\alpha \leq 0 \\
   \label{thm1:3}\tag{P1c}
   && &\forall i\colon \lambda_{i} \geq 0 \quad (\alpha \text{ free}), \\
   \intertext{or, equivalently, if and only if its dual problem, \linprogref{D1}, has a feasible solution:}
   \label{thm1:4}\tag{D1a}
   \linprogref{D1}  &&
   \max\quad & 0 \\
   \label{thm1:5}\tag{D1b}
   && \text{subject to}\quad & \forall f_i\in\mathcal{D}\colon \sum_{\omega \in \Omega} f_i(\omega)p(\omega) \geq 0\\
   \label{thm1:6}\tag{D1c}
   && & \sum_{\omega \in \Omega}p(\omega) = 1\\
   \label{thm1:7}\tag{D1d}
   && & \forall \omega\colon p(\omega) \geq 0.
\end{align}

\end{theorem}
Note that \linprogref{P1} is fully degenerate, and \linprogref{D1} is nearly fully degenerate.
Clearly, any feasible solution of \linprogref{D1} is also an optimal solution, since the objective function is constant.

\subsection{Reduced linear programming problem for checking avoiding sure loss}

When solving linear programs, several algorithms, such as the simplex and  the affine scaling methods, require all variables to be non-negative. Here we present alternative linear programming problems which are slightly smaller in dimension and have only non-negative variables. The following theorem is presented in \citep{2017:Nakharutai:Troffaes:Caiado} without a proof; we give a proof in the appendix.

\begin{theorem}\citep{2017:Nakharutai:Troffaes:Caiado}\label{thm:3} 
Choose any $\omega_0 \in \Omega$. The set $\mathcal{D}=\{f_1,\dots,f_n\}$ avoids sure loss if and only if the optimal value of \linprogref{P2} is zero:
\begin{align}
  \linprogref{P2}
  \tag{P2a}
  && \min \quad &  \sum_{i=1}^{n} \lambda_{i} f_{i}(\omega_0)+\alpha  \\
  \tag{P2b}
  \label{eq:19}
  && \text{subject to} \quad
  & \forall \omega \neq \omega_0\colon
  \sum_{i=1}^{n} (f_{i}(\omega_0)-f_{i}(\omega))\lambda_{i} +\alpha \geq 0 \\ 
  \tag{P2c}
  && & \forall i\colon\lambda_{i} \geq 0 \text{ and } \alpha \geq 0, 
  \intertext{or, equivalently, if and only if its dual problem, \linprogref{D2}, has a feasible solution:}
  \linprogref{D2}
  \tag{D2a}
  && \max \quad &  0 \\
  \tag{D2b}
  \label{eq:22}
  && \text{subject to}  \quad
  & \forall f_{i} \in \mathcal{D}\colon
  \sum_{\omega \neq \omega_0}(f_{i}(\omega_0)-f_{i}(\omega))p(\omega) \leq f_{i}(\omega_0)\\
  \tag{D2c}
  \label{eq:23}
  && & \sum_{\omega \neq \omega_0}p(\omega) \leq 1 \indent
  \\
  \tag{D2d}
  && & \forall \omega\colon p(\omega) \geq 0. 
\end{align}
\end{theorem}
\begin{proof}
See appendix.
\end{proof}

\linprogref{P2} is still fully degenerate, whilst \linprogref{D2} is no longer degenerate.

As primal optimality corresponds to dual feasibility \citep[p.~104]{2002:Goh:Yang}, if we choose any $\omega_0$ such that most values $f_i(\omega_0)$ are non-negative, then \linprogref{D2} will be closer to a dual feasible solution, and therefore \linprogref{P2} will also be closer to a primal optimal solution.
For instance, if there is an $\omega_0$ for which $f_{i}(\omega_0)\ge 0$ for all $i$, then we directly obtain a feasible solution of \linprogref{D2} by setting $p(\omega)=0$ for all $\omega\neq\omega_0$ \citep{2017:Nakharutai:Troffaes:Caiado}. The corresponding optimal solution of \linprogref{P2} is given by $\lambda_i =0$ for all $i$ and $\alpha =0$.

Next, we look at the simplex, the affine scaling and the primal-dual algorithms. We briefly explain how they work, and discuss improvements.

\section{Improving algorithms for solving linear programs}
\label{sec4}
\subsection{Simplex methods}
The simplex method is an iterative algorithm that needs an extreme point to start. If a problem can be written in the form:
\begin{equation}\tag{S}\label{eq:s}
\min c^\intercal x \text{ subject to } Ax + s = b,\ x\ge 0,\ s \ge 0,
\end{equation}
then provided that $b \geq 0$, we immediately obtain a starting extreme point by setting $s=b$ and $x=0$ \citep[\S 4.2]{2001:hillier}.
\Cref{eq:s} can be represented in a table.
The simplex algorithm performs row operations on this table to move between extreme points, improving the value of the objective function at each iteration, until we find the optimal value \citep[\S 4.4]{2001:hillier}.

The problem \linprogref{P2} can be easily written as \cref{eq:s} by negating
\cref{eq:19} and adding non-negative slack variables $s(\omega)$ \citep{2017:Nakharutai:Troffaes:Caiado}:
\begin{align}\label{eq:6.8}
  \linprogref{P3}
  \tag{P3a}
  && \min \quad
  &  \sum_{i=1}^{n} \lambda_i f_i(\omega_0)+\alpha  \\
  \tag{P3b}
  \label{eq:6.9}
  && \text{subject to} \quad & \forall \omega \neq \omega_0\colon
  \sum_{i=1}^{n} (f_i(\omega)-f_i(\omega_0))\lambda_i-\alpha+s(\omega)=0 \\ 
  \tag{P3c}
  \label{eq:6.10}
  && & \forall i\colon \lambda_i \geq 0,\
  \forall \omega\neq \omega_0\colon s(\omega) \geq 0\text{ and }\alpha\geq 0.
\end{align} 
Setting all $\lambda_i$, $ \alpha$ and $s(\omega)$ to zero
provides an initial extreme point.
Because \linprogref{P3} is fully degenerate,
this is also the only extreme point.

Full degeneracy also implies that, by the minimum ratio test \citep[p.~36]{1993:Fang:Puthenpura}, we can select any leaving basic variable at any point in the simplex method. Moreover, the value of any new entering basic variable is always zero, so the objective value never improves. This is bad news: if the same sequence of degenerate pivots associated with a non-optimal solution is repeated multiple times, then the simplex method is said to be \emph{cycling} and will never terminate \citep[\S 3.5]{1993:Fang:Puthenpura}. Fortunately, there are simple ways to avoid cycling, such as the lexicographic method or Bland's rule \citep[\S 3.6]{1993:Fang:Puthenpura}.

Once the problem of cycling is addressed, we may encounter another problem called \emph{stalling}, meaning that the method performs an exponentially long sequence of degenerate pivots \citep{1979:Cunningham}. Since there are no known efficient rules for choosing leaving variables for fully degenerate problems, we may end up visiting all extreme points. Efficiently solving fully degenerate problems is still an open problem \citep{Murty:2009}. 

We now consider the dual problem \linprogref{D2} and convert it into the form of \cref{eq:s}.
We first add non-negative slack variables to \cref{eq:22,eq:23} and obtain equality constraints. Since we want all the right hand side values to be non-negative, for every $j$ such that $f_j(\omega) < 0$, we multiply the corresponding constraint by $-1$ to make it non-negative and then add a non-negative artificial variable. In this case, the size of the linear programming problem is slightly bigger. We arrive at \cite{2017:Nakharutai:Troffaes:Caiado}:
{\small
\begin{align}
  \linprogref{D3}
  \tag{D3a}
  \label{eq:28}
  && \min \quad & \sum_{j \in N} v_j \\
  \tag{D3b}
  \label{eq:29}
  && \text{s.t.}\quad
  & \forall j \in N\colon
  \sum_{\omega \neq \omega_0}(f_j(\omega)-f_j(\omega_0))p(\omega)-s_j+v_j=-f_j(\omega_0)\\
  \tag{D3c}
  \label{eq:30} 
  && &\forall j \notin N\colon
  \sum_{\omega \neq \omega_0}(f_j(\omega_0)-f_j(\omega))p(\omega)+s_j=f_j(\omega_0)\\
  \tag{D3d}
  \label{eq:31}
  && & \sum_{\omega \neq \omega_0}p(\omega)+q=1 \\
  \tag{D3e}
  \label{eq:32}
  && & \forall\omega\neq \omega_0\colon p(\omega)\geq0,\,
  \forall j\colon s_j\geq0,\ v_j \geq0\text{ and } q \geq 0 
\end{align}
}with $N\coloneqq\{j\in N\colon f_{j}(\omega_0) < 0\}$.
An initial extreme point for \linprogref{D3} is given by $v_j = -f_j(\omega_0), s_j = 0 $ for all $j \in N $, $s_j = f_j(\omega_0)$ for all $j \notin N$, $p(\omega) = 0$  for all $\omega \neq \omega_0$ and $q = 1$.
If all $f_{j}(\omega_0) \geq 0$, then we have an immediate optimal solution.
The problem \linprogref{D3} is normally non-degenerate, except if $f_{j}(\omega_0) = 0$ for some $j$ \citep{2017:Nakharutai:Troffaes:Caiado}. 

To summarise, to check avoiding sure loss by the simplex method, we can solve either \linprogref{P3}, which is fully degenerate, or \linprogref{D3}, whose size is slightly larger.  Even though the simplex method may stall under degeneracy, in practice, it is still one of the most commonly used algorithms. This is why we treated it here. 

\subsection{Affine scaling methods}

The affine scaling method solves linear programs of the form of \cref{eq:p}. Given a starting interior feasible point, the method generates a sequence of interior feasible points which iteratively decrease the value of the objective function, until the improvement is small enough or unboundedness is detected \citep[\S 7.1.1]{1993:Fang:Puthenpura}.

The affine scaling method can solve \linprogref{P3} as it is already in the form of \cref{eq:p}. Similar to the simplex method, degeneracy can affect the performance of the affine scaling method. \Citet{1992:Tsuchiya:Muramatsu} show that this can be overcome by limiting the step-size of the algorithm. 

However, unlike the simplex method, \cref{lem:1} can be applied to the affine scaling method. Specifically, when the affine scaling solves \linprogref{P3}, the method can stop as soon as it finds a negative value for the objective function.

The dual problem \linprogref{D2} can be written in the form of \cref{eq:p} by adding non-negative slack variables:
\begin{align}
  \linprogref{D4} &&
  \tag{D4a}
  \label{eq:9}
  \max \quad &  0 \\
  \tag{D4b}
  \label{eq:10}
  && \text{subject to}\quad
  & \forall f_{i} \in \mathcal{D}\colon
  \sum_{\omega \neq \omega_0}(f_{i}(\omega_0)-f_{i}(\omega))p(\omega) +t_i = f_{i}(\omega_0)\\
  \tag{D4c}
  \label{eq:11}
  && & \sum_{\omega \neq \omega_0}p(\omega) +q = 1 \\ 
  \tag{D4d}
  \label{eq6.18}
  && & \forall\omega\neq \omega_0\colon
  p(\omega)\geq0,\ \forall i\colon t_i\geq0,\ q \geq 0.
\end{align}
Although we can also solve \linprogref{D3} by the affine scaling method, we solve \linprogref{D4} as it has fewer artificial variables.

As the affine scaling method requires an initial interior feasible point, we normally need to solve two linear programming problems: one to find a starting interior feasible point, and another one to solve the original problem with this starting point. 

A starting interior feasible solution can be found as follows (see \citep[\S 7.1.2]{1993:Fang:Puthenpura} for more details).
Consider the constraints $Ax=b$ and $x  \geq  0$.
Choose any point $x^0>0$ and calculate $z=b-Ax^0$.
If $z=0$, then  $x^0$ is an interior feasible solution of the original problem.
Otherwise, solve
\begin{equation}
  \label{eq:affinescaling:phase1}
  \tag{P'}
  \min\gamma\text{ subject to }Ax +z\gamma = b,\ x\ge 0,\ \gamma \ge 0
\end{equation}
by the affine scaling method, using $[x~~\gamma] =[x^0~~1]$ as a starting point (this point is an interior feasible solution of \cref{eq:affinescaling:phase1}). If an optimal solution is $[x^*~~\gamma^*]$ with $\gamma^*=0$, then $x^*$ is an interior feasible solution of the original problem. Otherwise, there is no feasible solution. 

The good news for us is that, for solving either \linprogref{P3} or \linprogref{D4}, we only need to solve a single linear programming problem.
For \linprogref{D4}, this is because every interior feasible point is also an optimal solution, so we only need to solve \cref{eq:affinescaling:phase1}.
For \linprogref{P3}, due to the structure of the problem,
we can immediately write down an interior feasible point in closed form,
so we do not need to solve \cref{eq:affinescaling:phase1}.

Let us explain how \cref{eq:affinescaling:phase1} looks like
for \linprogref{D4},
following (with a slight improvement here)
\citet{2017:Nakharutai:Troffaes:Caiado}.
Let $\Omega\setminus \{\omega_0\}= \{\omega_1, \dots,\omega_m\}$.
Consider, for the moment, an arbitrary
\begin{equation}\label{eq7.50}
  x^0 =
  \begin{bmatrix}
    p^0 (\omega_1) & \cdots & p^0 (\omega_m) &  t^0_1 &  \cdots & t^0_n &  q^0
  \end{bmatrix}
  > 0
\end{equation}
and define $[r~~z]\coloneqq b - Ax^0$, so
\begin{align}
  \label{eq:34}
  r_i&\coloneqq h_i - t_i^0\\
  \label{eq:35}
  z &\coloneqq
  1 - \left( \sum_{\omega \neq \omega_0}p^0(\omega) +q^0\right) \\
  \intertext{where}
  h_i&\coloneqq   f_{i}(\omega_0) - \sum_{\omega \neq \omega_0}(f_{i}(\omega_0)-f_{i}(\omega))p^0(\omega).
\end{align}
If we choose $q^0 = p^0(\omega) = 1/ |\Omega|$ for all $\omega \neq \omega_0$, then $z=0$.
Choose $t_i^0=1$ (or any other strictly positive value)
for all $i$ where $h_i\le 0$.
Finally, choose $t_i^0=h_i$ for all $i$ where $h_i>0$,
so all corresponding $r_i$ are zero.
So, \cref{eq:affinescaling:phase1} becomes:
\begin{align}\label{eq:36}   
  \linprogref{D4'}
  \tag{D4'a}
  && \min \quad & \gamma \\
  \tag{D4'b}
  && \text{s.t.}\quad
  & \forall i\colon
  \sum_{\omega \neq \omega_0}(f_{i}(\omega_0)-f_{i}(\omega))p(\omega) +t_i +r_i\gamma = f_{i}(\omega_0)\\
  \tag{D4'c}
  && & \sum_{\omega \neq \omega_0}p(\omega) +q = 1 \\
  \tag{D4'd}
  && & \forall\omega\neq \omega_0\colon p(\omega)\geq0,\,
  \forall i\colon t_i\geq0,\, q\geq0\text{ and } \gamma \geq 0,
\end{align}
with an initial interior feasible point as constructed. Note that for simplicity in our implementation, we choose $t_{i}^{0} = 1$ for all $i$.
If the optimal solution of \linprogref{D4'} has $\gamma^*=0$, then we will have found an interior feasible solution for \linprogref{D4} (and therefore also an optimal solution for \linprogref{D4}), and so $\mathcal{D}$ avoids sure loss; otherwise, there is no feasible solution and $\mathcal{D}$ incurs sure loss.

For \linprogref{P3}, we simply calculate a starting interior feasible point using \cref{thm:4} below, with $\lambda_i^0=1$.
\begin{theorem}\citep[modified]{2017:Nakharutai:Troffaes:Caiado}
\label{thm:4}
An interior feasible point of the following system of linear constraints 
\begin{align} \label{eq:56}
  \quad & \forall j\in\{1,\dots,m\}\colon
  \sum_{i=1}^{n} a_{ij}\lambda_i- \alpha  +s_j = b_j  \\
  \label{eq:57}
  & \forall i\colon\lambda_i\geq 0,\ \forall j\colon s_j\geq 0,\ \alpha\geq 0
\end{align}
is given by setting $\lambda_i=\lambda_i^{0}$ for some arbitrary $\lambda_i^{0}>0$,
$\alpha = 1+\max\{0, -\delta\}$ with 
\begin{equation}
\delta\coloneqq\min_{j}\left\{b_j-\sum_{i=1}^{n}a_{ij}\lambda_i^{0}\right\},
\end{equation}
and $s_{j} =b_j- \sum_{i=1}^{n}a_{ij}\lambda_i^{0}+\alpha $ .
\end{theorem}
\begin{proof}
We must show that \cref{eq:56} is satisfied, and that all variables are strictly positive.

Clearly, \cref{eq:56} is satisfied by our choice of $s_j$,
all $\lambda_i=\lambda_i^{0}>0$, and $\alpha\ge 1>0$.
Finally, note that also all $s_j>0$ because
\begin{equation}\label{eqff}
s_{j}=b_j- \sum_{i=1}^{n}a_{ij}\lambda_i^{0}+\alpha \geq\delta + \alpha\geq \delta +1- \delta> 0,
\end{equation}
where we used the definitions of $\delta$ and $\alpha$ respectively.
\end{proof}

To conclude, to check avoiding sure loss with the affine scaling method, we either solve \linprogref{P3} or \linprogref{D4'}. In either case, we have a closed form initial interior feasible point. In addition, we can apply \cref{lem:1} as an extra stopping rule to detect unboundedness when solving \linprogref{P3}. However, we need to take care to limit the step-size due to degeneracy.

Next, we look at another interior-point method for which we can also apply \cref{lem:1,thm:4}, but which does not have a limitation on the step-size and which has faster convergence as observed in practice.

\subsection{Primal-dual methods}

The primal-dual method is an iterative algorithm which finds an optimal solution by solving \cref{eq:p,eq:d} simultaneously. At each iteration, the method solves the following system:
\begin{equation}\label{eq:11.1}
 \begin{bmatrix}
Ax - b\\ A^\intercal y +t - c \\ x^\intercal t    
\end{bmatrix} = 0,\, x\ge 0,\, t   \geq 0 \qquad (y \text{ free})
\end{equation}
whilst keeping the variables $x$ and $t$ positive. Theoretically, given an initial interior feasible point $(x,y,t)$ where $x>0$ and $t>0$, the method can generate a sequence of interior feasible points such that $x^\intercal t$ gets closer and closer to zero (by duality, the solution is optional when $x^\intercal t=0$; the method simply exploits this fact). However, in practical implementations, keeping $(x,y,t)$ in the feasible region is very difficult due to numerical issues \citep[\S 7.3]{1993:Fang:Puthenpura}. 

Therefore, in practical implementations, the primal-dual method starts with an arbitrary point $(x,y,t)$ where $x>0$ and $t>0$ and generates points that converge to a feasible optimal solution. The method will stop when the primal residual $Ax - b$, dual residual $A^\intercal y +t - c$, and  duality gap  $x^\intercal t$ are small enough, or when unboundedness in either the primal or the dual is detected. Although this modified version of the algorithm has no known convergence proof, it works extremely well in practice \citep[\S 7.3]{1993:Fang:Puthenpura}. 

The problem \linprogref{P3} is already in the form of \cref{eq:p}. Its dual is \citep{2017:Nakharutai:Troffaes:Caiado}:
\begin{align}
  \tag{D5a}
  \label{eq7.57}
  \linprogref{D5} &&
  \max \quad &  0 \\
  \tag{D5b}
  \label{eq7.58}
  && \text{s. t.} \quad &
  \forall f_{i} \in \mathcal{D}\colon
  \sum_{\omega \neq \omega_0 }(f_{i}(\omega )-f_{i}(\omega_0))v(\omega)+t_i =  f_{i}(\omega_0)\\
  \tag{D5c}
  \label{eq7.59}
  && & q -  \sum_{\omega \neq \omega_0}v(\omega) = 1 \\
  \tag{D5d}
  \label{eq7.60}
  && & \forall \omega \neq \omega_0\colon v(\omega) +p(\omega) = 0\\
  \tag{D5e}
  \label{eq7.61}
  && &  \forall i\colon t_i \geq 0,\ \forall \omega\neq \omega_0\colon p(\omega) \geq0 \text{ and } q \geq 0.
\end{align}
Because $-v(\omega)=p(\omega) \geq 0$, \linprogref{D5} is equivalent to \linprogref{D4}, as expected. The primal-dual method solves \linprogref{P3} and \linprogref{D5} simultaneously.

\Cref{thm:4} provides an initial interior feasible point for \linprogref{P3}. However, there is no closed form feasible point for \linprogref{D5} (if we had, then we immediately would have found an optimal solution). In this case, a starting point of \linprogref{D5} can be $q^0 =p^0(\omega)= 1/|\Omega|$, $v^0(\omega)= - 1/|\Omega|$, and $t^0_i= 1$ for all $i$ \citep{2017:Nakharutai:Troffaes:Caiado}.

Remind that we can apply \cref{lem:1} to \linprogref{P3} only if we can keep all iterative points in the feasible region. Although we start \linprogref{P3} with a feasible point, the next points do not necessarily remain in the feasible region due to numerical rounding errors. Therefore, it is good practice to calculate the primal residual and only apply \cref{lem:1} if this error is neglegible \citep{2017:Nakharutai:Troffaes:Caiado}.

Now, consider solving the problem \linprogref{D4'} by the primal-dual method. In this case, the dual of \linprogref{D4'} can be written in the form of \cref{eq:d} as follows:
\begin{align}
  \tag{P4'a}
  \label{eq:63}
  \linprogref{P4'}
  && \max \quad &  \sum_{i=1}^{n} \lambda_{i} f_{i}(\omega_0)+\alpha  \\ 
  \tag{P4'b}
  \label{eq:64}
  && \text{s. t.} \quad & \forall \omega \neq \omega_0\colon \sum_{i=1}^{n} (f_{i}(\omega_0)-f_{i}(\omega))\lambda_{i} +\alpha +s(\omega) =  0 \\ 
  \tag{P4'c}
  \label{eq:65}
  && & \forall i\colon \lambda_i + u_i = 0\\
  \tag{P4'd}
  \label{eq:66}    
  && & \alpha + \beta = 0\\
  \tag{P4'e}
  \label{eq:67}
  && & \sum_{i=1}^{n} r_i\lambda_i +\mu = 1\\
  \tag{P4'f}
  \label{eq:68}
  && & \forall i\colon u_i\ge 0,\ \forall \omega\neq \omega_0\colon s(\omega)\ge 0,\ \beta\ge 0 \text{ and } \mu \geq 0. 
\end{align} 
To find an initial interior feasible point of \linprogref{P4'}, first choose $\lambda_i<0$ such that $\sum_{i=1}^{n} r_i\lambda_i < 1$. The $u_i>0$ are then fixed by \cref{eq:65}, and $\mu$ is fixed by \cref{eq:67}. Note that $\mu = 1- \sum_{i=1}^{n} r_i\lambda_i > 0$ by construction. Substituting $\alpha = - \beta$ into \cref{eq:64}, we can then apply \cref{thm:4} to find interior feasible values for $\beta$ and $s(\omega)$ for all $\omega \neq \omega_0$. Unfortunately we cannot apply \cref{lem:1} to \linprogref{P4'}, because the problem is no longer fully degenerate.

In the next section, we explain how we can generate random sets of desirable gambles that either avoid or do not avoid sure loss. In \cref{sec8}, we will then benchmark our three methods on those randomly generated sets.


\section{Generating sets of desirable gambles via coherent lower previsions}\label{sec7}

In this section, we first give algorithms for generating coherent previsions (\cref{alg:Exp}), polyhedral lower previsions (\cref{alg:poly}) and linear-vacuous mixtures (\cref{alg:lin-vac}), as mentioned earlier. We then discuss how they can be used to generate sets of desirable gambles for benchmarking.

\begin{breakablealgorithm}
\caption{Generate a coherent prevision}
\label{alg:Exp}
\begin{algorithmic}
\REQUIRE Set of outcomes $\Omega$
\ENSURE Coherent prevision $P$ on $\mathcal{L}(\Omega)$
 \begin{enumerate}
\item[Stage 1.] Generate a probability mass function $p$ as follows:
\begin{enumerate}
\item For each $\omega$, sample $r_\omega$ uniformly from $(0,1)$.
\item For each $\omega$, set $p(\omega)\coloneqq \dfrac{\ln r_\omega}{\sum_{\omega \in \Omega} \ln r_\omega}$.
\end{enumerate}
\item[Stage 2.] Generate a coherent prevision $P$
\begin{enumerate}
\item For any $f \in \mathcal{L}(\Omega)$, $P(f) \coloneqq E_p(f)$ as in \cref{eq:8.3}.
\end{enumerate}
\end{enumerate}
\end{algorithmic}
\end{breakablealgorithm}

\begin{breakablealgorithm}
\caption{Generate a polyhedral lower prevision}
\label{alg:poly}
\begin{algorithmic}
\REQUIRE \parbox[t]{\textwidth}{Set of outcomes $\Omega$\\
$k$ coherent previsions: $Q_1,\dots,Q_k$ (e.g. obtained by \cref{alg:Exp})}
\ENSURE Polyhedral lower prevision $\underline{P}$ on $\mathcal{L}(\Omega)$
 \begin{enumerate}
\item[Stage 1.] For any $f \in \mathcal{L}(\Omega)$, 
$\underline{P}(f)\coloneqq \min_{j=1}^k\{Q_{j}(f)\}$.
\end{enumerate}
\end{algorithmic}
\end{breakablealgorithm}

\begin{breakablealgorithm}
\caption{Generate a linear-vacuous mixture}
\label{alg:lin-vac}
\begin{algorithmic}
\REQUIRE \parbox[t]{\textwidth}{Set of outcomes $\Omega$\\
$\delta \in (0,1)$ (e.g. sample  $\delta$ uniformly from $(0,1)$) \\
Coherent prevision $Q$ (e.g. generated by \cref{alg:Exp})}
\ENSURE Linear-vacuous mixture $\underline{P}$
 \begin{enumerate}
\item[Stage 1.] For any $f \in \mathcal{L}(\Omega)$,
$\underline{P}(f)\coloneqq (1-\delta)Q(f) + \delta \inf f$.
\end{enumerate}
\end{algorithmic}
\end{breakablealgorithm}

We now explain how to generate sets of desirable gambles that avoid sure loss from any given coherent lower previsions. As we will see, if this coherent lower prevision is more generic, then the generated set of desirable gambles will also be more generic. In addition to the analysis presented here, note that we also generated sets of desirable gambles from lower previsions that avoid sure loss but that are not coherent. However, we found no practical difference; see discussion at the end of \cref{sec8}.
 
We start by generating a coherent lower prevision $\underline{E}$ on $\mathcal{L}(\Omega)$ e.g.\ through one of the above algorithms. Next, we generate a finite subset $\mathcal{K}$ of $\mathcal{L}(\Omega)$ and set $\underline{P}(f)\coloneqq\underline{E}(f)$ for all $f \in\mathcal{K}$.
$\underline{P}$ is coherent too because it is the restriction of a coherent lower prevision \citep[p.~58]{2014:troffaes:decooman::lower:previsions}.
Therefore, the set $\mathcal{D}\coloneqq\{f-\underline{P}(f)\colon f \in \mathcal{K}\}$ avoids sure loss:
\begin{breakablealgorithm}
\caption{Generate a set of desirable gambles that avoids sure loss}
\label{alg:ASL}
\begin{algorithmic}
\REQUIRE\parbox[t]{\textwidth}{Set of outcomes $\Omega$\\
Number of desirable gambles $n\coloneqq|\mathcal{D}|$\\
Coherent lower prevision  $\underline{E}$ on $\mathcal{L}(\Omega)$}
\ENSURE Finite set of desirable gambles $\mathcal{D}$ that avoids sure loss
 \begin{enumerate}
\item[Stage 1.] Generate $\{f_j\colon j\in\{1,\dots, n\}\}$:\\
for each $\omega$ and $j$, sample $f_j(\omega)$ uniformly from $(0,1)$.
\item[Stage 2.] For each $i\in\{1,\dots,n\}$, calculate $\underline{E}(f_i)$.
\item[Stage 3.] Set $\mathcal{D}\coloneqq \{f_i - \underline{E}(f_i)\colon i\in\{1,\dots,n\}\}$.
\end{enumerate}
\end{algorithmic}
\end{breakablealgorithm}
Which type of coherent lower prevision should we use to generate sets of desirable gambles? To answer this question, first we look at $\mathcal{M}_{\underline{E}}$ and its extreme points, for various classes of $\underline{E}$:
\begin{enumerate}
\item[(i)] Vacuous lower prevision: $\mathcal{M}_{\underline{E}} = \Delta(\Omega)$ and its extreme points are all $0-1$ valued probabilities.
\item[(ii)] Coherent previsions: $\mathcal{M}_{\underline{E}} = \ext\mathcal{M}_{\underline{E}} =\{p\}$, $p \in \Delta(\Omega)$.
\item[(iii)] Polyhedral lower previsions: as in \cref{eq:8.8}, when $M$ is finite, $\mathcal{M}_{\underline{E}}$ is a polyhedron and has a finite set of extreme points.
\item[(iv)] Linear-vacuous mixtures: for $p_0 \in \Delta(\Omega)$ and $\delta > 0$, $\mathcal{M}_{\underline{E}} = \{(1-\delta)p_0 + \delta p,p\in\Delta(\Omega) \}$ and $\ext\mathcal{M}_{\underline{E}} = \{(1-\delta)p_0 + \delta p,\ p \text{ is a } 0-1\text{ valued probability} \}$.
\end{enumerate}

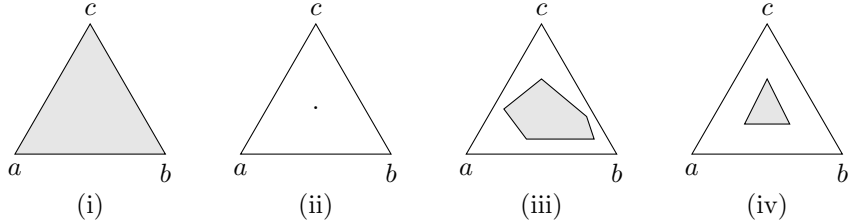
\begin{figure}[ht]
\centering
\begin{tikzpicture}
\filldraw[draw=black, fill=gray!20] (0,0) node[anchor=north]{$a$}
  -- (2,0) node[anchor=north]{$b$}
  -- (1,1.732) node[anchor=south]{$c$}
  -- cycle;
\draw (1,-1) node[anchor=south]{(i)};  

\draw (3,0) node[anchor=north]{$a$}
  -- (5,0) node[anchor=north]{$b$}
  -- (4,1.732) node[anchor=south]{$c$}
  -- cycle;
\draw (4,0.62) node{.};  
\draw (4,-1) node[anchor=south]{(ii)};  

\draw (6,0) node[anchor=north]{$a$}
  -- (8,0) node[anchor=north]{$b$}
  -- (7,1.732) node[anchor=south]{$c$}
  -- cycle;
\filldraw[draw=black, fill=gray!20] (6.5,0.6)-- (7,1)-- (7.6,0.5)--(7.7,0.2)-- (6.8,0.2)-- cycle;
\draw (7,-1) node[anchor=south]{(iii)};  

\draw (9,0) node[anchor=north]{$a$}
  -- (11,0) node[anchor=north]{$b$}
  -- (10,1.732) node[anchor=south]{$c$}
  -- cycle;
\filldraw[draw=black, fill=gray!20] (10,1)-- (9.7,0.4)--(10.3,0.4)-- cycle;
\draw (10,-1) node[anchor=south]{(iv)};  
\end{tikzpicture}
\caption{Simplex representation of different $\mathcal{M}_{\underline{E}}$ for $\Omega=\{a,b,c\}$: (i) vacuous, (ii) prevision, (iii) polyhedral lower prevision, (iv) linear-vacuous mixture.}\label{fig:1}
\end{figure}
\Cref{fig:1} shows examples of $\mathcal{M}_{\underline{E}}$ associated with different coherent lower previsions.
For polyhedral lower previsions, the number of extreme points is arbitrary (but finite). For linear-vacuous mixtures, the number of extreme points is limited to the number of outcomes, and the shape of its credal set is fixed (up to scale and translation).

When we generate $\mathcal{D}=\{f - \underline{P}(f)\colon f \in \dom\underline{P}\}$ as in \cref{alg:ASL}, the credal set associated with $\mathcal{D}$ is given by
\begin{equation}
\mathcal{M}_{\mathcal{D}} = \{p \in \Delta(\Omega)\colon \forall f \in \mathcal{D},\  E_{p}(f) \geq 0 \}.
\end{equation}
$\mathcal{M}_{\mathcal{D}}$ and $\mathcal{M}_{\underline{E}}$ are related as follows:
\begin{corollary}\label{co:1} Let
 $\underline{E}$ be a coherent lower prevision on $\mathcal{L}(\Omega)$, let $\underline{P}$ be a restriction of $\underline{E}$ to a finite domain, and let $\mathcal{D}\coloneqq\{f - \underline{P}(f)\colon f \in \dom \underline{P}\}$. Then:
\begin{enumerate}[(i)]
\item\label{co:1:subset} $\mathcal{M}_{\underline{E}} \subseteq \mathcal{M}_{\mathcal{D}}$.
\item\label{co:1:equal} If $\mathcal{M}_{\underline{E}}$ is a polyhedron, then there is a finite set $\mathcal{K}\subseteq\mathcal{L}(\Omega)$ such that if $ \dom \underline{P} =\mathcal{K}$, then $\mathcal{M}_{\mathcal{D}} = \mathcal{M}_{\underline{E}}$.
\end{enumerate}
\end{corollary}
\begin{proof}
(i). We find that 
\begin{align}\label{eq5.16}
  \mathcal{M}_{\underline{E}}
  & = \bigcap_{f \in \mathcal{L}(\Omega)} \left\lbrace p \in \Delta(\Omega)\colon E_p(f) \geq \underline{E}(f) \right\rbrace \\
  & = \bigcap_{f \in \mathcal{L}(\Omega)} \left\lbrace p \in \Delta(\Omega)\colon E_p(f-\underline{E}(f)) \geq 0 \right\rbrace \\
  & \subseteq  \bigcap_{f \in \dom\underline{P}} \left\lbrace p \in \Delta(\Omega)\colon E_p(f-\underline{P}(f)) \geq 0 \right\rbrace = \mathcal{M}_{\mathcal{D}}.
\end{align}

(ii). By (i), we only need to show that $\mathcal{M}_{\mathcal{D}} \subseteq \mathcal{M}_{\underline{E}}$.
Since $\mathcal{M}_{\underline{E}} $ is a polyhedron, it is an intersection of a finite number of half-spaces. Therefore there exists an $n \in \mathbb{N}$, vectors $f_1,\dots,f_n$ and numbers $\alpha_1,\dots,\alpha_n$ such that 
\begin{equation}\label{eq:72}
\mathcal{M}_{\underline{E}} = \bigcap_{i=1}^{n}\{p \in \Delta(\Omega)\colon p \cdot f_i \geq \alpha_i\}
\end{equation}
where `$\cdot$' denotes the dot product.
Note that
\begin{equation}
  \underline{E}(f_i)
  =\min_{p \in\mathcal{M}_{\underline{E}} }p\cdot f_{i}
  =\min_{p}\{p\cdot f_{i}\colon \forall j,\ p\cdot f_{j} \geq \alpha_j\}\geq\alpha_i.
\end{equation}
Set $\mathcal{K} = \{f_1,\dots,f_n\}$ and $\dom \underline{P} = \mathcal{K}$.
Then,
\begin{align}
  \mathcal{M}_{\mathcal{D}}
  &= \bigcap_{i=1}^{n}\{p \in \Delta(\Omega)\colon p \cdot f_i \geq \underline{E}(f_i)\}
  \\
  &\subseteq \bigcap_{i=1}^{n}\{p \in \Delta(\Omega)\colon p \cdot f_i \geq \alpha_i\}
  = \mathcal{M}_{\underline{E}}.
\end{align}
\end{proof}

Consequently, if we want $\mathcal{M}_\mathcal{D}$ to have a sufficient number of extreme points, we should generate  $\mathcal{D}$ using either a polyhedral lower prevision, or at the very least using a linear-vacuous mixture.

\begin{figure}
\centering
\begin{tikzpicture}
\filldraw[draw=black, fill=gray!20] (0,0) node[anchor=north]{$a$}
  -- (2,0) node[anchor=north]{$b$}
  -- (1,1.732) node[anchor=south]{$c$}
  -- cycle;
\draw (1,-1) node[anchor=south]{(i)};  
\draw (0.02,1.3)--(0,1.45)--(2,2)--(2.03,1.9);

\draw (3,0) node[anchor=north]{$a$}
  -- (5,0) node[anchor=north]{$b$}
  -- (4,1.732) node[anchor=south]{$c$}
  -- cycle;
\draw (4,0.62) node{.};  
\draw (4,-1) node[anchor=south]{(ii)};  
\draw (3.02,0.35)--(3,0.5)--(5,0.75)--(5,0.6);
\draw (3.1,0.95)--(3.1,0.8)--(5.2,0.4)--(5.2,0.55);
\draw (3.4,1.95)--(3.5,2)--(4.35,-.3)--(4.25,-.34);

\draw (6,0) node[anchor=north]{$a$}
  -- (8,0) node[anchor=north]{$b$}
  -- (7,1.732) node[anchor=south]{$c$}
  -- cycle;
\filldraw[draw=black, fill=gray!20] (6.5,0.6)-- (7,1)-- (7.6,0.5)--(7.7,0.2)-- (6.8,0.2)-- cycle;
\draw (7,-1) node[anchor=south]{(iii)};  
\draw (5.8,0.35)--(5.8,.2)--(8.1,.2)--(8.1,0.35);
\draw (6.1,1.45)--(6.0,1.35)--(7.1,-.25)--(7.2,-.17);
\draw (5.8,-.1)--(5.75,0)--(7.55,1.45)--(7.65,1.35);
\draw (6.4,1.3)--(6.5,1.4)--(8.2,0)--(8.15,-.1);
\draw (7.1,1.77)--(7.2,1.8)--(7.8,-.1)--(7.7,-.15);

\draw (9,0) node[anchor=north]{$a$}
  -- (11,0) node[anchor=north]{$b$}
  -- (10,1.732) node[anchor=south]{$c$}
  -- cycle;
\filldraw[draw=black, fill=gray!20] (10,1)-- (9.7,0.4)--(10.3,0.4)-- cycle;
\draw (10,-1) node[anchor=south]{(iv)};  
\draw (8.8,0.55)--(8.8,0.4)--(11.2,0.4)--(11.2,0.55);
\draw (9.6,-.3)--(9.4,-.2)--(10.5,2)--(10.65,1.9);
\draw (9.35,1.9)--(9.47,2)--(10.62,-.2)--(10.5,-0.3);
\end{tikzpicture}
\caption{Constructing $\mathcal{M}_{\underline{E}}$ by finite half-spaces for $\Omega=\{a,b,c\}$: (i) vacuous, (ii) prevision, (iii) polyhedral lower prevision, (iv) linear-vacuous mixture.}\label{fig:2}
\end{figure}
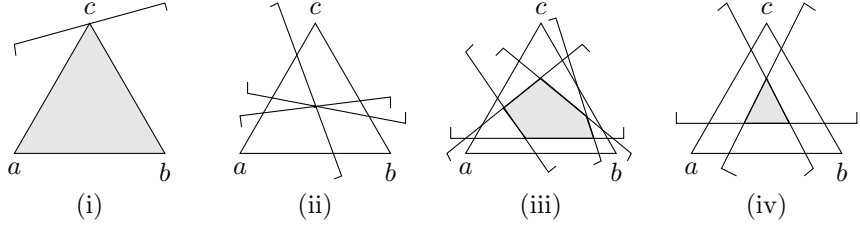

\Cref{fig:2} visualises the construction of the proof, for various classes of lower previsions.
\Cref{co:1} implies that we do not necessarily gain new extreme points as we add more and more gambles to $\mathcal{D}$. Therefore, \cref{alg:ASL} may not be a good way for generating $\mathcal{D}$ if our aim is to keep adding new extreme points with every gamble we add.
The algorithm in the next section addresses precisely this issue.

\section{Sequentially generating sets of desirable gambles}\label{sec7next}

Consider a set $\mathcal{E}=\{f_1,\dots,f_n\}$ that avoids sure loss.
How can we add another gamble, say $f$, for which $ \mathcal{E} \cup \{f\}$ either still avoids sure loss, or not?

\subsection{Generating sets of desirable gambles that do not avoid sure loss}

Given a gamble $g$, we first find the range of values for $\alpha$ such that $ \mathcal{E} \cup \{g-\alpha\}$ avoids sure loss. By the condition of avoiding sure loss in \cref{def:1}, we obtain a constraint on the values of $\alpha$ as follows: for all $ n \in \mathbb{N}$, all  $\lambda_{1}, \dots,\lambda_{n} \geq 0$, and all $f_{1}, \dots,f_{n} \in \mathcal{E}$,
\begin{equation}
\sup_{\omega \in \Omega}\left(\sum_{i=1}^{n} \lambda_if_i(\omega) +g(\omega) \right) \geq \alpha.
\end{equation}
The infimum of this upper bound is precisely $\overline{E}_{\mathcal{E}}(g)$. So, we proved:
\begin{corollary}\label{cor:1}
Let $\mathcal{E}$ avoid sure loss and let $g \in  \mathcal{L}(\Omega)$. $ \mathcal{E} \cup \{g-\alpha\}$ avoids sure loss if and only if $\alpha \leq \overline{E}_{\mathcal{E}}(g)$.
\end{corollary}
Hence, if we set $\alpha > \overline{E}_{\mathcal{E}}(g)$, then $ \mathcal{E} \cup \{g-\alpha\}$ does not avoid sure loss:
\begin{breakablealgorithm}
\caption{Generate a set of desirable gambles that does not avoid sure loss}
\label{alg:NASL}
\begin{algorithmic}
\REQUIRE \parbox[t]{\textwidth}{Set of outcomes $\Omega$\\
  A set of desirable gambles $\mathcal{E}$ that avoids sure loss\\
  $\delta>0$ (e.g. sample  $\delta$ uniformly from $(0,1)$)}
\ENSURE A set of desirable gambles $\mathcal{D}$ that does not avoid sure loss.
 \begin{enumerate}
\item[Stage 1.] For each $\omega\in \Omega$, sample $g(\omega)$ uniformly from $(0,1)$.
\item[Stage 2.]  Solve the following linear program:
\begin{align}
\min \quad & \beta  \\ 
    \text{s. t.} \quad & \forall \omega \in \Omega\colon \sum_{f_i\in\mathcal{E}} f_{i}(\omega)\lambda_{i} -\beta \leq  -g(\omega),\quad
   \lambda_i \geq 0 \quad(\beta \text { free}). 
\end{align}
\item[Stage 3.]  Set $\mathcal{D}\coloneqq \mathcal{E} \cup\{  g-\beta-\delta\}$.
\end{enumerate}
\end{algorithmic}
\end{breakablealgorithm}

Note that sets of gambles generated by \cref{alg:NASL} only contain a single gamble that violate consistency. Therefore, they are the most computationally challenging sets to detect not avoiding sure loss. Consequently, they are the most suitable sets for benchmarking, as any measurable improvement on these cases implies an at least as large improvement on any simpler cases.

\subsection{Generating sets of desirable gambles that avoid sure loss}

Consider a coherent lower prevision $\underline{Q}$. The set $\mathcal{E}\coloneqq\{f-\underline{Q}(f)\colon f \in \dom\underline{Q}\}$ then avoids sure loss. Let $g \in \mathcal{L}(\Omega)\setminus \dom\underline{Q}$. By \cref{cor:1}, we know that the larger set $\mathcal{D}:= \mathcal{E}\cup \{g-\alpha\}$ still avoids sure loss as long as $\alpha \leq \overline{E}_{\mathcal{E}}(g)$. 
Note that the number of extreme points can decrease after adding $\{g-\alpha\}$, as shown in \cref{fig:3}.
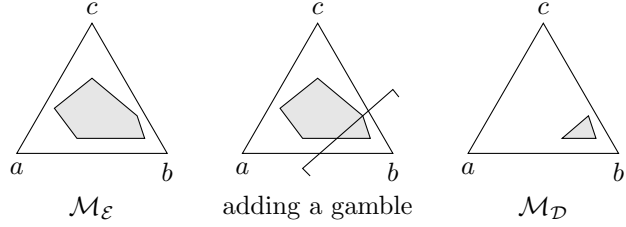
\begin{figure}[ht]
\centering
\begin{tikzpicture}
\draw (0,0) node[anchor=north]{$a$}
  -- (2,0) node[anchor=north]{$b$}
  -- (1,1.732) node[anchor=south]{$c$}
  -- cycle;
\filldraw[draw=black, fill=gray!20] (0.5,0.6)-- (1,1)-- (1.6,0.5)--(1.7,0.2)-- (0.8,0.2)-- cycle;
\draw (1,-1) node[anchor=south]{$\mathcal{M}_{\mathcal{E}}$}; 

\draw (3,0) node[anchor=north]{$a$}
  -- (5,0) node[anchor=north]{$b$}
  -- (4,1.732) node[anchor=south]{$c$}
  -- cycle;
\filldraw[draw=black, fill=gray!20] (3.5,0.6)-- (4,1)-- (4.6,0.5)--(4.7,0.2)-- (3.8,0.2)-- cycle;
\draw (4,-1) node[anchor=south]{adding a gamble}; 
\draw (5.08,0.75)--(5,0.85)--(3.8,-0.2)--(3.9,-0.3);

\draw (6,0) node[anchor=north]{$a$}
  -- (8,0) node[anchor=north]{$b$}
  -- (7,1.732) node[anchor=south]{$c$}
  -- cycle;
\draw (7,-1) node[anchor=south]{$\mathcal{M}_{\mathcal{D}}$}; 
\filldraw[draw=black, fill=gray!20]  (7.6,0.5)--(7.7,0.2)-- (7.25,0.2)-- cycle;

\end{tikzpicture}
\caption{Simplex representation of a credal sets after adding a gamble.}\label{fig:3}
\end{figure}

How should we choose $\alpha$ to avoid reducing the number of extreme points?
If $\underline{P}$ is a coherent extension
of $\underline{Q}$ to $\dom\underline{Q}\cup\{g\}$,
then $\mathcal{M}_{\underline{P}}$ must have
at least as many extreme points as $\mathcal{M}_{\underline{Q}}$,
because, by coherence,
for every $f\in\dom\underline{Q}$,
$\underline{P}(f)$ and $\underline{Q}(f)$ must be achieved
at some extreme point of $\mathcal{M}_{\underline{P}}$
and $\mathcal{M}_{\underline{Q}}$, respectively \citep[p.~126]{1991:walley}.
But because $\underline{P}(f)=\underline{Q}(f)$ for all these gambles $f$,
it cannot be that $\mathcal{M}_{\underline{P}}$ has fewer extreme
points than $\mathcal{M}_{\underline{Q}}$,
because otherwise $\underline{P}(f)>\underline{Q}(f)$
for at least one gamble $f$.
Hence, the number of extreme points does not decrease if we keep coherence.

\begin{theorem}(adapted from \citep[p.~126]{1991:walley})
Let $\underline{Q}$ be a coherent lower prevision and let $g\in \mathcal{L}(\Omega)\setminus\dom\underline{Q}$.
Let $\underline{P}$ be an extension of $\underline{Q}$ to $\dom\underline{Q} \cup\{g\}$. Then $\underline{P}$ is coherent if and only if $\underline{P}(g)\in [\underline{E}_{\underline{Q}}(g), \underset{^\sim}{E}{}_{\underline{Q}}(g)]$, where 
\begin{multline}
\underset{^\sim}{E}{}_{\underline{Q}}(g) \coloneqq \inf_{f_i\in\dom\underline{Q},\,\lambda_i\ge 0} \Bigg\lbrace \\
\max_{\omega \in \Omega} \left( g(\omega) + \sum_{i=1}^{n} \lambda_i ( f_i(\omega)-\underline{Q}(f_i)))- \lambda_0(f_0(\omega) -\underline{Q}(f_0)) \right)\Bigg\rbrace.
\end{multline}
\end{theorem}

\Cref{fig:4} shows ranges of avoiding sure loss and coherence of $g$ given a coherent lower prevision $\underline{Q}$.
\begin{figure}
\centering
\begin{tikzpicture}
\draw (1,0) -- (2,0)node[anchor=north]{$\underline{E}_{\underline{Q}}(g)$} --(5,0)node[anchor=north]{$\underset{^\sim}{E}{}_{\underline{Q}}(g)$} --(8,0)node[anchor=north]{$\overline{E}_{\underline{Q}}(g)$}
  -- (10,0);
\draw   (2.2,0.3)--(2,0.3);
\draw (2,-0.1)--(2,0.3);
  \draw   (4.8,0.3)--(5,0.3);
\draw (5,-0.1)--(5,0.3);
\draw (3.5,0) node[anchor=south]{coherence}; 
\draw (6,0.4) node[anchor=south]{avoiding sure loss}; 
\draw   (8,0.7)--(8,0);
\draw   (8,0.7) -- (7.8,0.7);
\draw[->]   (4.2,0.7)--(3,0.7);
\end{tikzpicture}
\caption{Ranges of avoiding sure loss and coherence}\label{fig:4}
\end{figure}
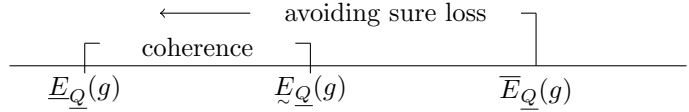

We can calculate $\underline{E}_{\underline{Q}}(g)$ by solving a single linear programming problem. However, $\underset{^\sim}{E}{}_{\underline{Q}}(g)$ cannot be obtained by solving just a single linear programming problem:
we have to solve a separate linear programming problem for every $f_0\in\dom\underline{Q}$,
where each separate linear program is very similar to and has the same size as
the linear program for calculating $\overline{E}_{\underline{Q}}(g)$.

So, instead of finding $\underset{^\sim}{E}{}_{\underline{Q}}(g)$, we calculate $\underline{E}_{\underline{Q}}(g)$ and $\overline{E}_{\underline{Q}}(g)$ (or $\overline{E}_{\mathcal{E}}(g)$ where $\mathcal{E}= \{f-\underline{Q}(f)\colon f \in \dom \underline{Q}\}$) since for each of them, we solve only one linear programming problem. Next, we choose a very small number $\delta$ and set $\underline{P}(g)\coloneqq (1-\delta)\underline{E}_{\underline{Q}}(g)+\delta\overline{E}_{\underline{Q}}(g)$. Then, $\underline{P}(g)$ is slightly larger than $\underline{E}_{\underline{Q}}(g)$, but $\underline{P}(g)$ is still less than $\overline{E}_{\underline{Q}}(g)$. Therefore $\mathcal{E}\cup \{g-\underline{P}(g)\}$ still avoids sure loss as we want. 
This approach is summarised in \cref{alg:ASL2}.
Because only the natural extension of $\underline{Q}$ is used, we only need $\underline{Q}$ to avoid sure loss.

\begin{breakablealgorithm}
\caption{Generate a set of desirable gambles that avoids sure loss}
\label{alg:ASL2}
\begin{algorithmic}
\REQUIRE \parbox[t]{\textwidth}{Set of outcomes $\Omega$\\
A lower prevision $\underline{Q}$ that avoids sure loss (e.g. Algorithms \ref{alg:Exp}, \ref{alg:poly} or \ref{alg:lin-vac})\\
A set $\mathcal{E}=\{f-\underline{Q}(f)\colon f\in \dom\underline{Q}\}$ that avoids sure loss \\ $\delta\in(0,1)$}
\ENSURE A larger set of desirable gambles $\mathcal{D}$ that avoids sure loss
 \begin{enumerate}
\item[Stage 1.] For each $\omega\in \Omega$, sample $g(\omega)$ uniformly from $(0,1)$.
\item[Stage 2.] Calculate $\overline{E}_{\underline{Q}}(g)$ by solving
\begin{align}
\min \quad & \beta  \\ 
    \text{subject to}\quad & \forall \omega \in \Omega\colon \sum_{i=1}^n (f_{i}(\omega)-\underline{Q}(f_i))\lambda_{i} -\beta \leq  -g(\omega) \\
      & \lambda_i \geq 0,\, f_{i}(\omega)-\underline{Q}(f_i) \in \mathcal{E}\qquad(\beta \text{ free})
\end{align}
\item[Stage 3.] Calculate $\underline{E}_{\underline{Q}}(g)$ by solving
\begin{align}
\max \quad &\gamma \\ 
\text{subject to} \quad
& \forall \omega \in \Omega\colon \sum_{i=1}^{n}(f_{i}(\omega)-\underline{Q}(f_i)) \lambda_{i} +\gamma \leq  g(\omega) \\
& \lambda_i \geq 0,\,  f_{i}(\omega)-\underline{Q}(f_i) \in \mathcal{E}\qquad(\gamma \text { free}). 
\end{align}
\item[Stage 4.] Set $\underline{P}(g)\coloneqq(1-\delta)\underline{E}_{\underline{Q}}(g)+\delta\overline{E}_{\underline{Q}}(g)$.
\item[Stage 5.] Set $\mathcal{D}\coloneqq \mathcal{E} \cup\{  g-\underline{P}(g)\}$.
\end{enumerate}
\end{algorithmic}
\end{breakablealgorithm}

Note that we will not use \cref{alg:ASL2} in our benchmarking. Instead, we use a combination of \cref{alg:poly} and \cref{alg:ASL}, which generates constraints via a set of probability mass functions, as this is computationally faster. We state \cref{alg:ASL2} for the sake of completeness, as an alternative algorithm for generating constraints directly.

\section{Numerical results}\label{sec8}

To benchmark our theoretical result, in this section, we generate two types of random sets of desirable gambles: sets that avoid sure loss, and sets that do not. For each type, we consider $|\mathcal{D}| = 2^i$ for $i \in \{1,2,\dots, 8\}$ and $|\Omega| = 2^j$ for $j \in \{1,2,\dots, 8\}$. Random sets that avoid sure loss are generated as follows:

\begin{enumerate}
\item We use \cref{alg:Exp} to generate $k$ coherent previsions. We fixed $k = 2^5$, as we observed that varying $k$ has little impact on the results.
\item From these $k$ coherent previsions, we use \cref{alg:poly} to generate a polyhedral lower prevision.
\item We use \cref{alg:ASL}, with this polyhedral lower prevision, to generate a random set that avoids sure loss.
\end{enumerate}
Next, starting from a set $\mathcal{E}$ that avoids sure loss, we generate a set that does not avoid sure loss using \cref{alg:NASL} with $\delta = 0.05$. 

For each random set of desirable gambles, we use the algorithms for
checking avoiding sure loss discussed earlier in the paper. \Cref{table:1} gives an overview of these different algorithms. Note that the primal-dual method simultaneously solves the primal and the dual problems.

\begin{table}
\centering
\begin{tabular}{|c|c|c|c|}
\hline
\multicolumn{1}{|c|}{\multirow{2}{*}{Linear programs}} & \multicolumn{3}{c|}{Methods}                                                 \\ \cline{2-4} 
\multicolumn{1}{|c|}{}                  & \multicolumn{1}{c|}{Simplex} & \multicolumn{1}{c|}{Affine scaling} & \multicolumn{1}{c|}{Primal-dual} \\ \hline
\linprogref{P3} & \checkmark & \checkmark &$\checkmark$\\
\hline
\linprogref{D3} & \checkmark &  &\\
\hline
\linprogref{D4'} &  & \checkmark &$\checkmark$\\
\hline   
\end{tabular}
\caption{List of different methods for checking avoiding sure loss. Note that the primal-dual method also solves the dual problem simultaneously, e.g. \linprogref{D5} for \linprogref{P3} and \linprogref{P4'} for \linprogref{D4'}.}
\label{table:1}
\end{table}

To compare these three methods, we wrote our own implementation of the affine scaling and the primal-dual methods. We used an implementation of the revised simplex method written by \citet{2002:Strang}. Indeed, the revised simplex method is mathematically equivalent to the standard simplex method, but is much more efficient and numerically stable as it applies sparse matrix manipulations \citep[\S 3.7]{1993:Fang:Puthenpura}. We used the revised simplex method to solve both problems \linprogref{P3} and \linprogref{D3}. For the affine scaling method, a standard version is used for solving \linprogref{D4'}, while an improved version, which includes the extra stopping criterion and our mechanism for calculating feasible starting points, is used to solve \linprogref{P3}. For the primal-dual method, an improved version that has the extra stopping criterion and the mechanism for calculating feasible starting points is used for solving \linprogref{P3} and \linprogref{D5}, and another improved version that has only the mechanism for calculating feasible starting points is used to solve \linprogref{D4'} and \linprogref{P4'}.

\begin{figure}
  \centering
  \newcolumntype{C}{>{\centering\arraybackslash} m{0.4\linewidth} }
  \begin{tabular}{m{1em}CC}
  &
  Avoiding sure loss
  &
  Not avoiding sure loss
  \\
  \rotatebox[origin=l]{90}{$|\mathcal{D}| = 2^2$}
  &
  \includegraphics[width=\hsize]{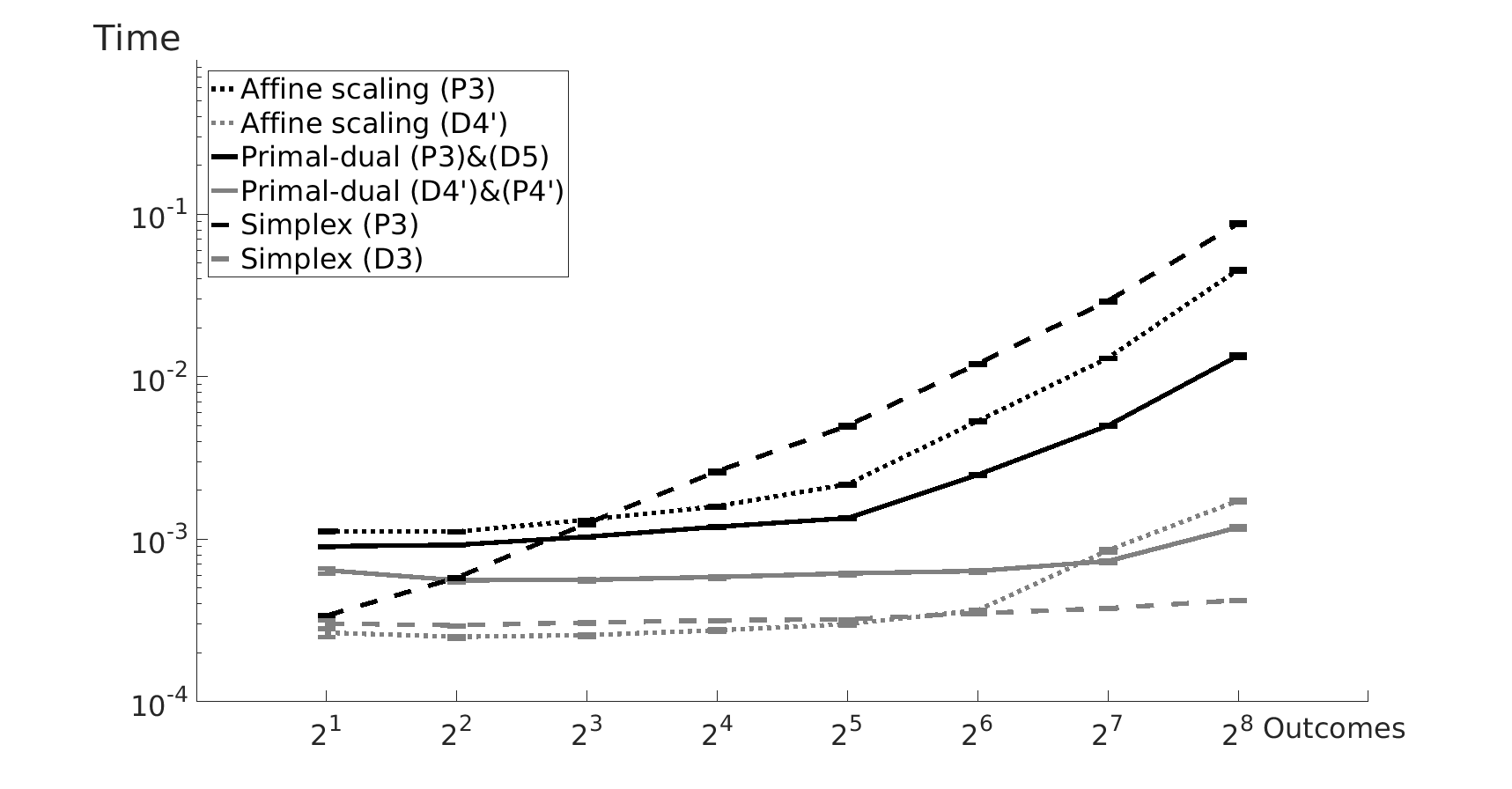}
  &
  \includegraphics[width=\hsize]{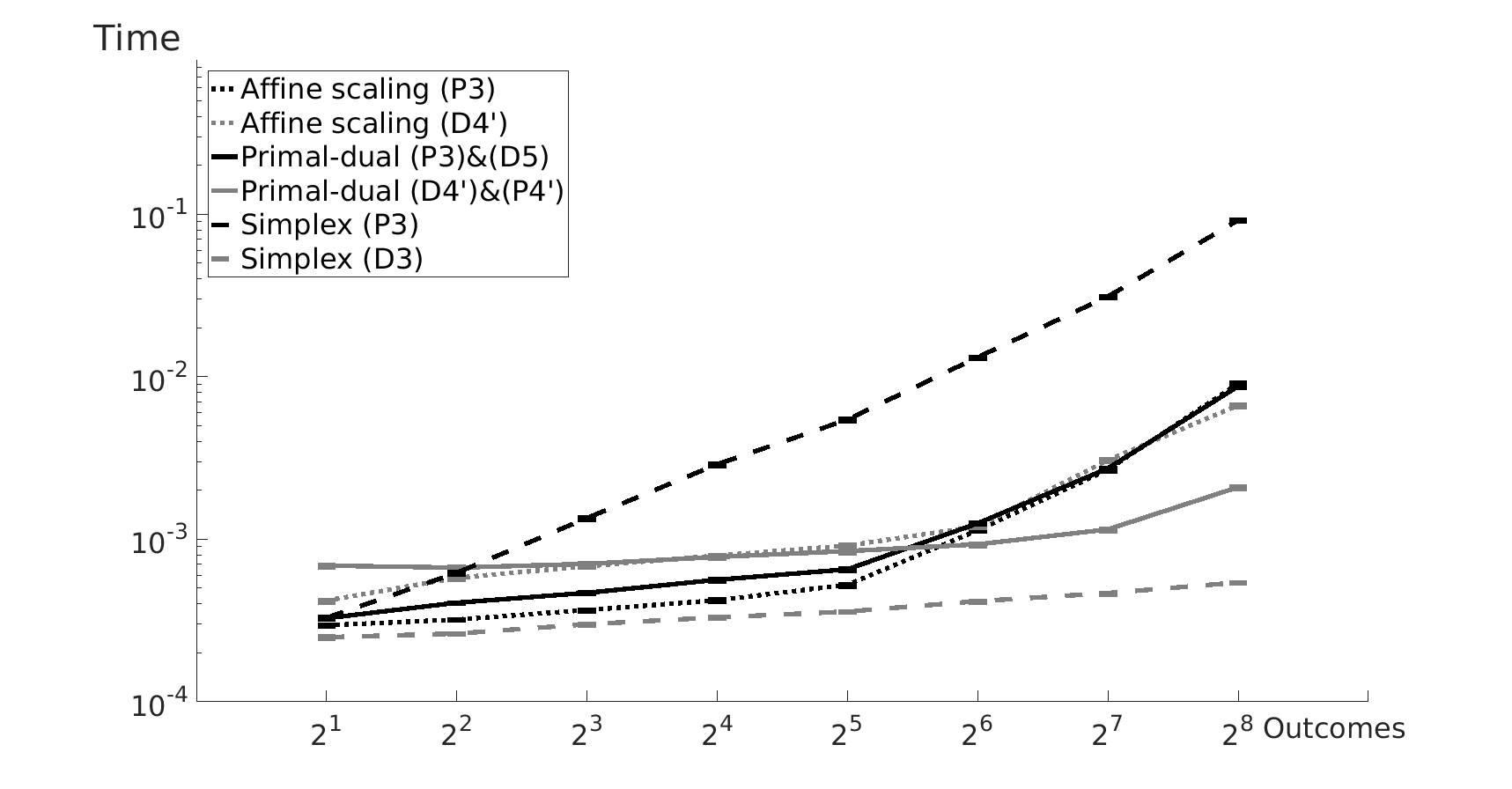}
  \\
  \rotatebox[origin=l]{90}{$|\mathcal{D}| = 2^4$}
  &
  \includegraphics[width=\hsize]{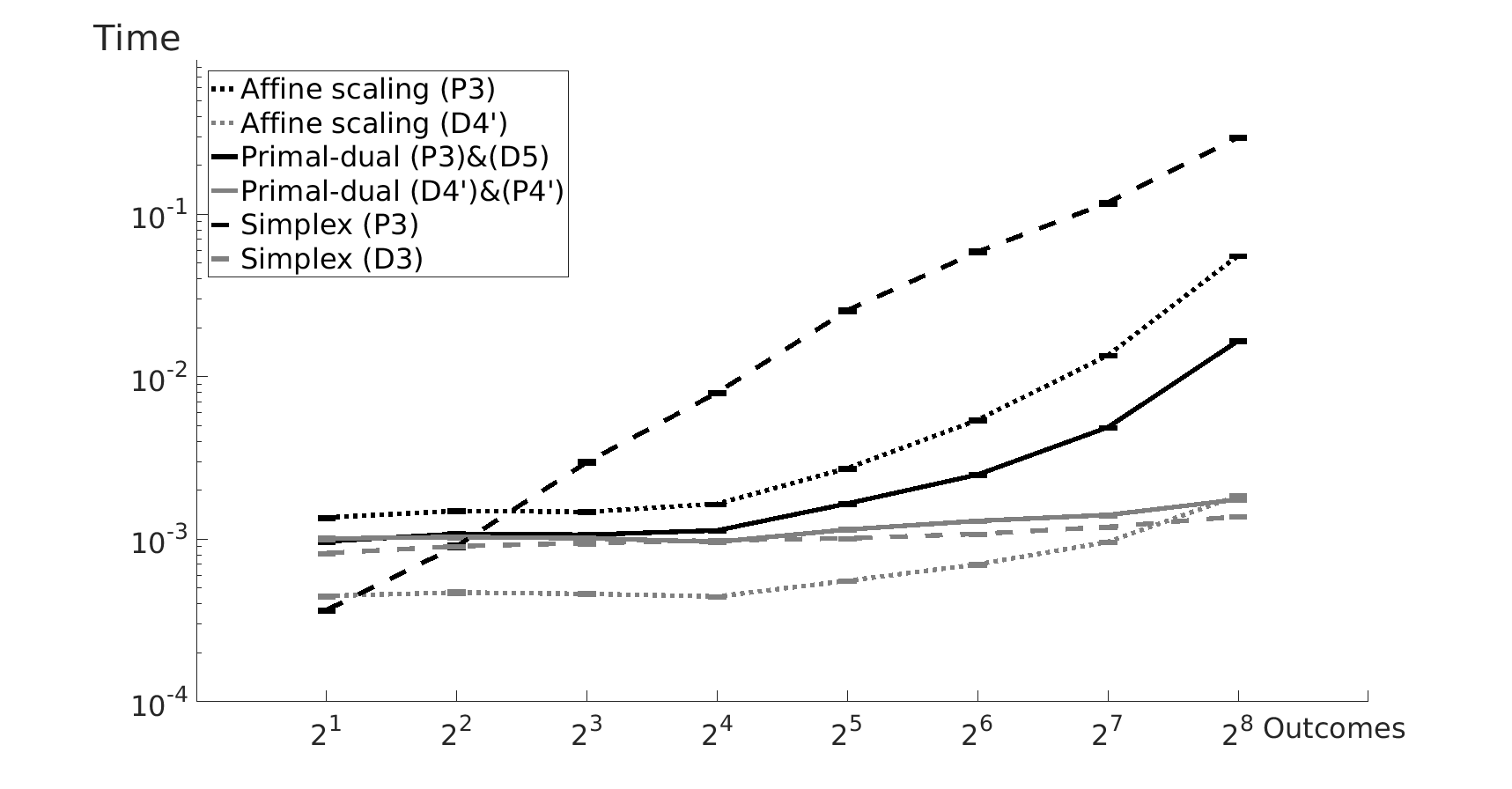}
  &
  \includegraphics[width=\hsize]{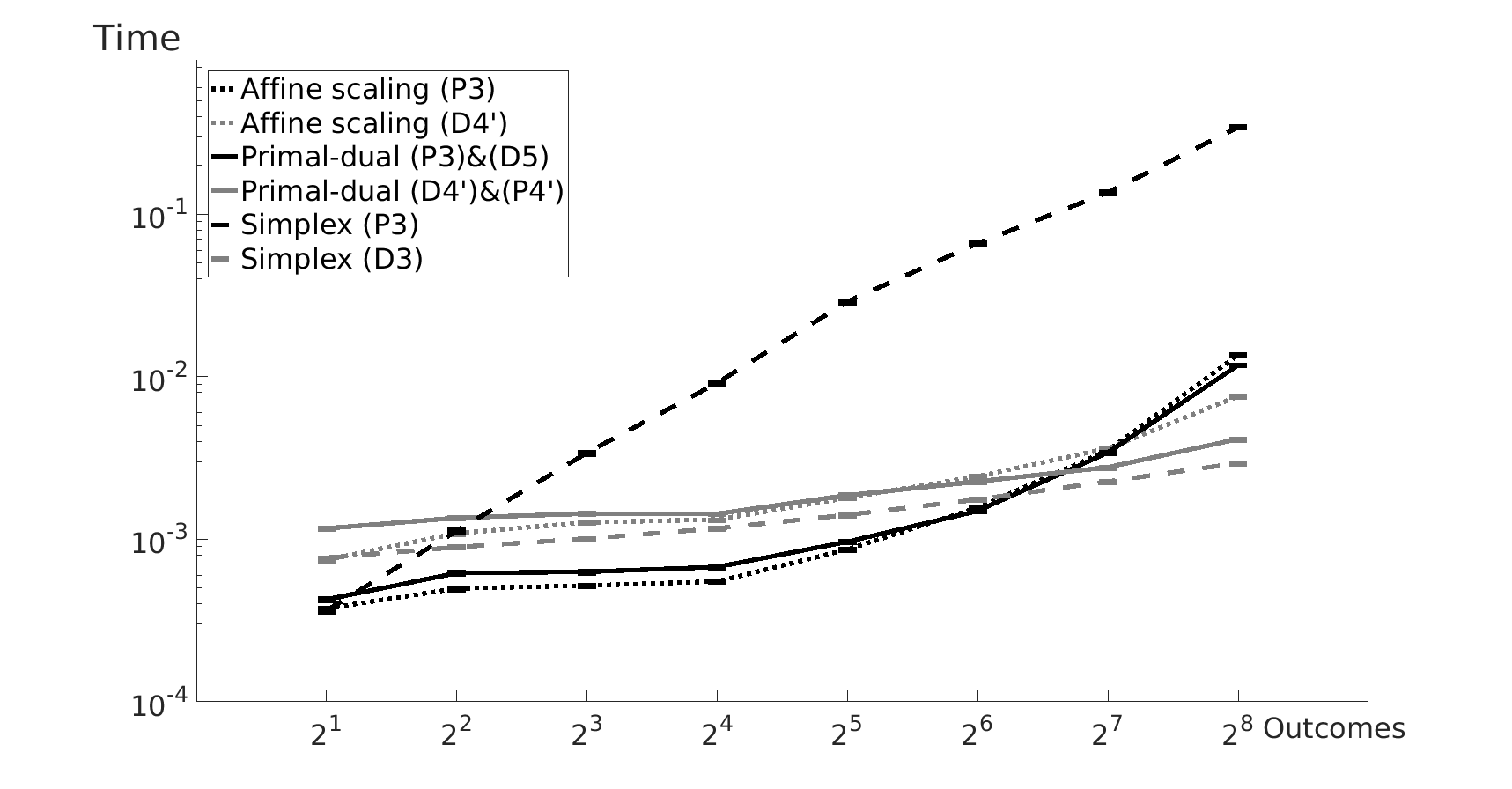}
  \\
  \rotatebox[origin=l]{90}{$|\mathcal{D}| = 2^6$}
  &
  \includegraphics[width=\hsize]{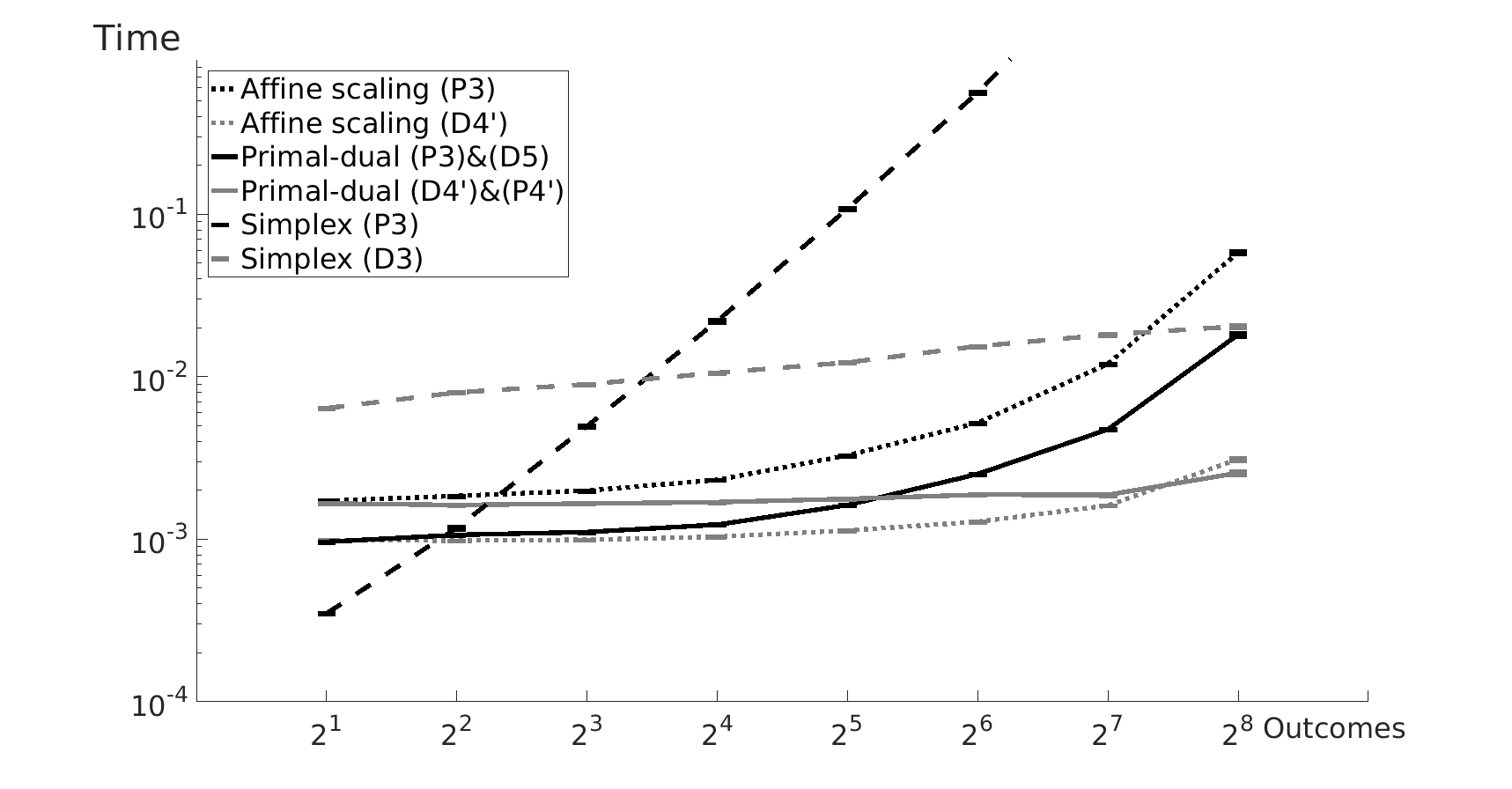}
  &
  \includegraphics[width=\hsize]{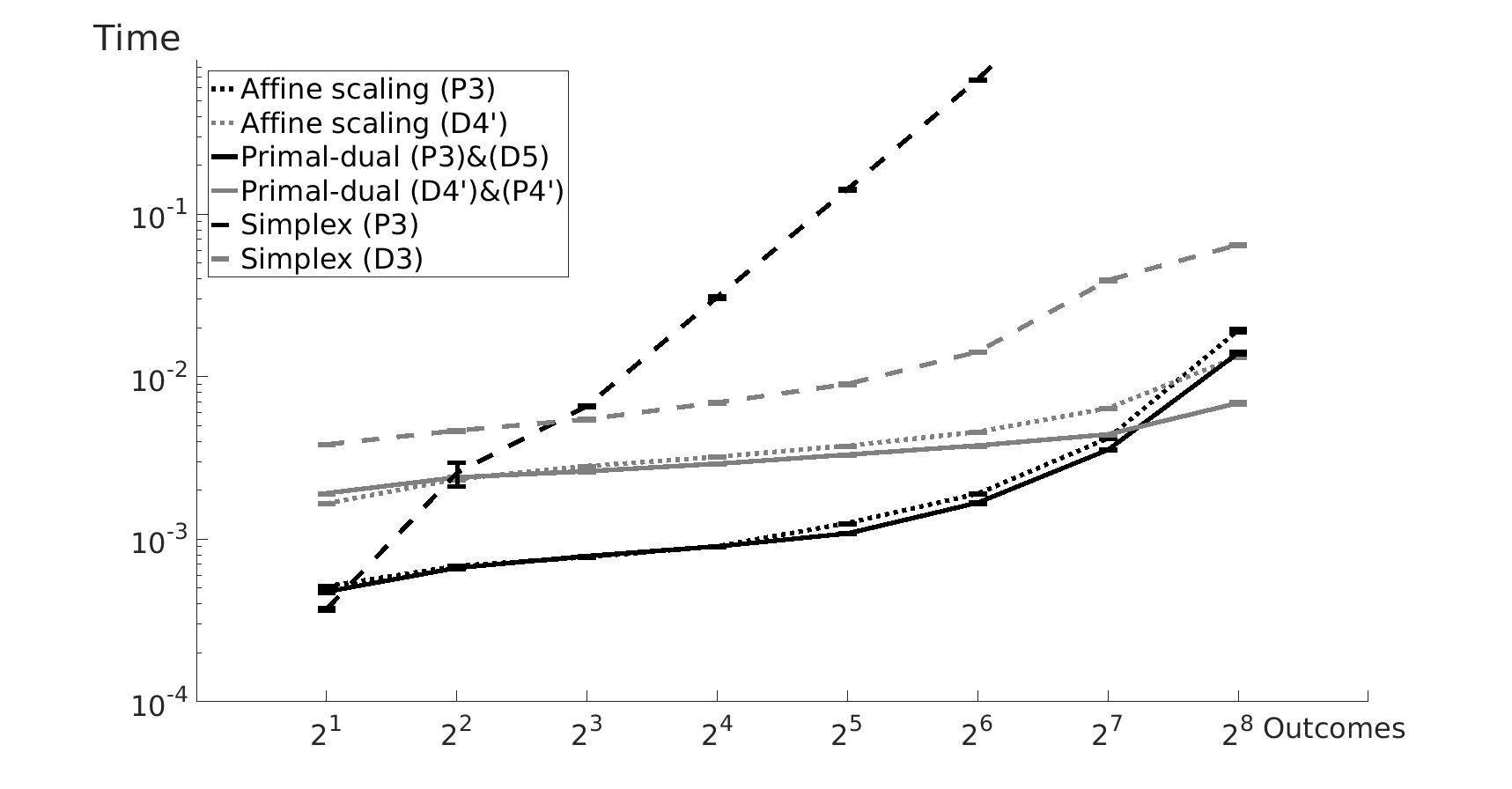}
  \\
  \rotatebox[origin=l]{90}{$|\mathcal{D}| = 2^8$}
  &
  \includegraphics[width=\hsize]{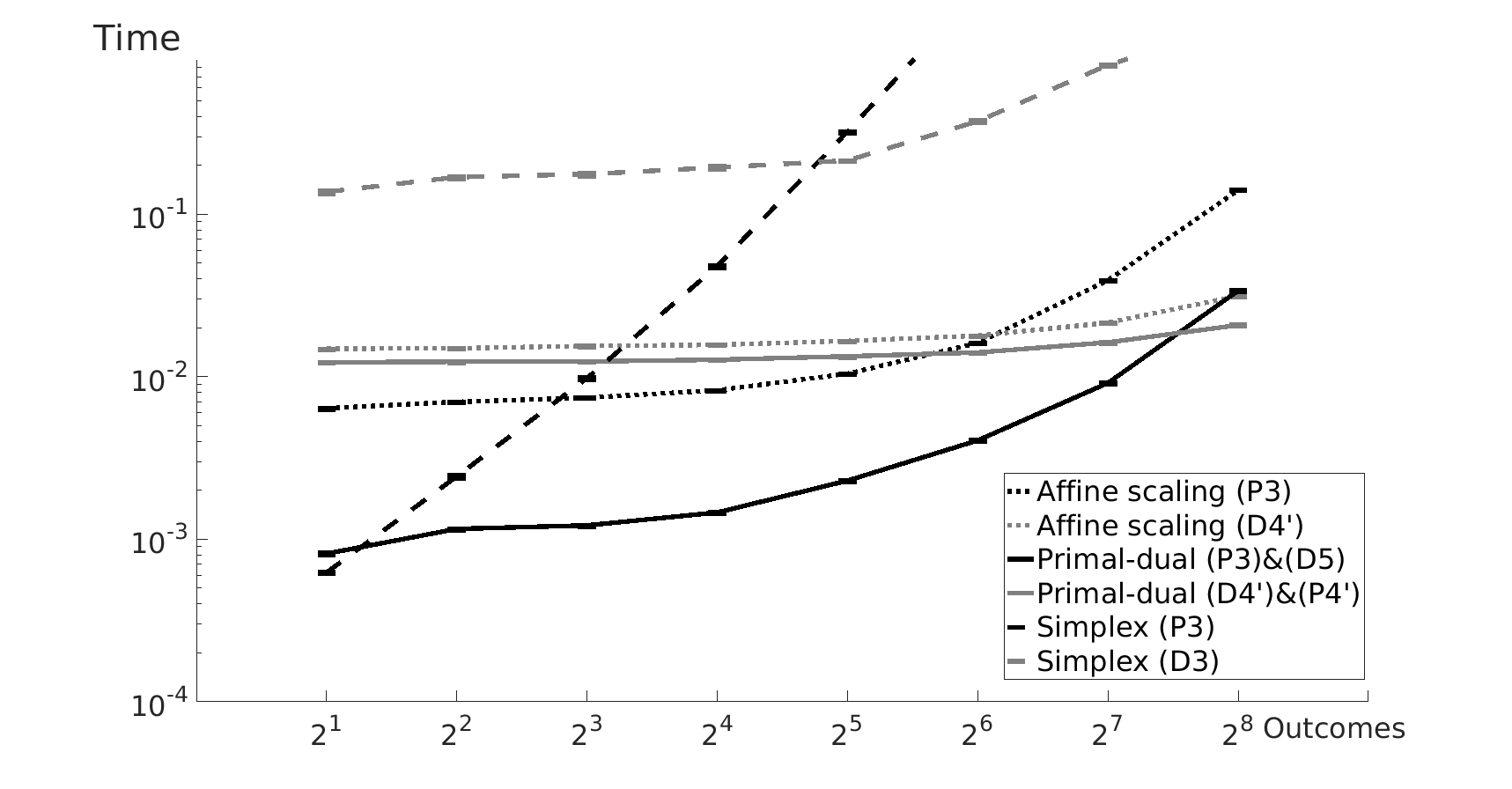}
  &
  \includegraphics[width=\hsize]{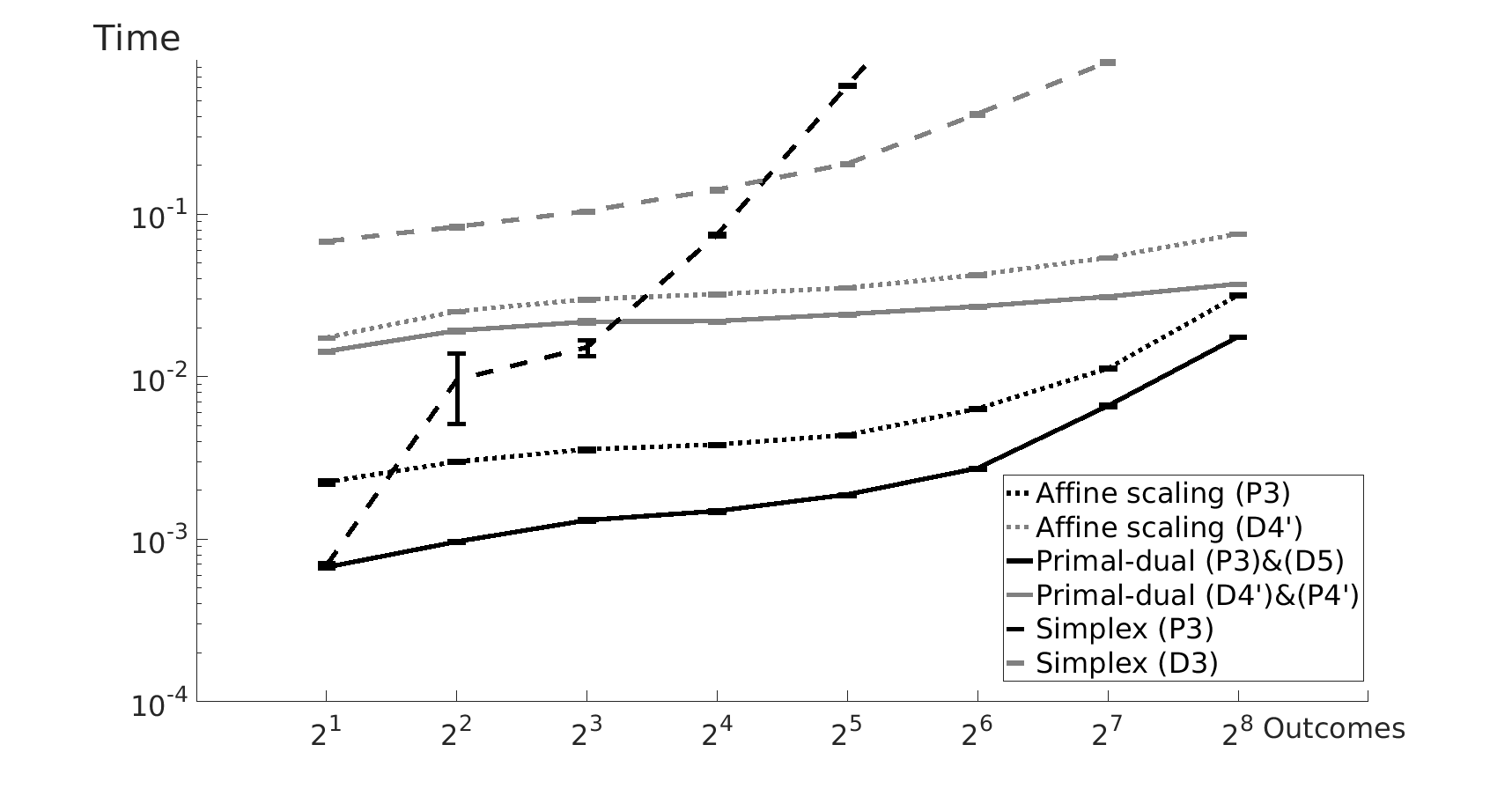}
\end{tabular}
\caption{Comparison plots of the average computational time for three methods. The left column avoids sure loss and the right column does not.  Each row represents a different number of desirable gambles with vary the number of outcomes.  The labels indicate linear programs solved by different methods.}
\label{fig:6}
\end{figure}

\begin{figure} 
  \centering
  \newcolumntype{C}{>{\centering\arraybackslash} m{0.4\linewidth} }
  \begin{tabular}{m{1em}CC}
  &
  Avoiding sure loss
  &
  Not avoiding sure loss
  \\
  \rotatebox[origin=l]{90}{$|\Omega| = 2^2$}
  &
  \includegraphics[width=\hsize]{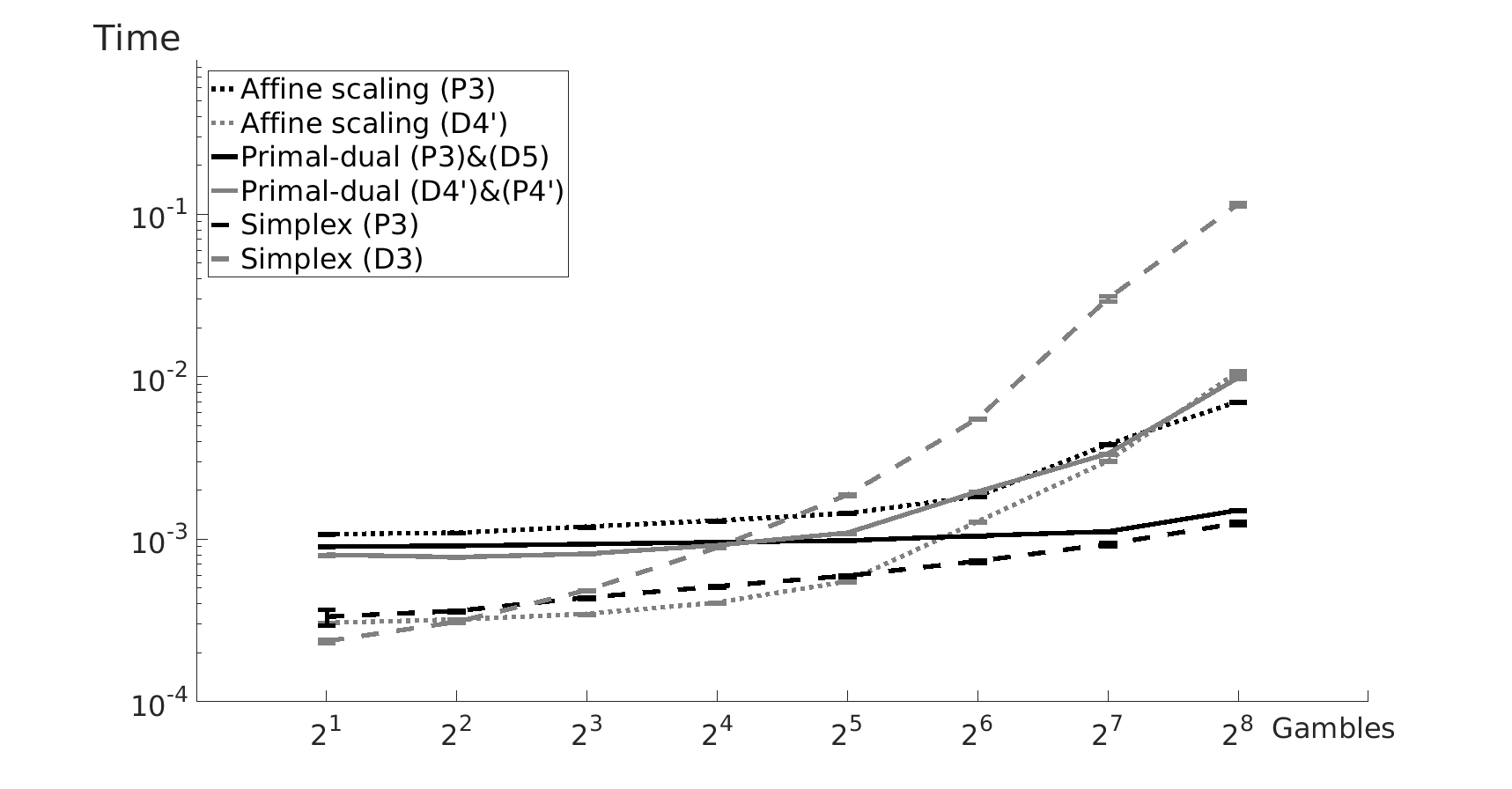}
  &
  \includegraphics[width=\hsize]{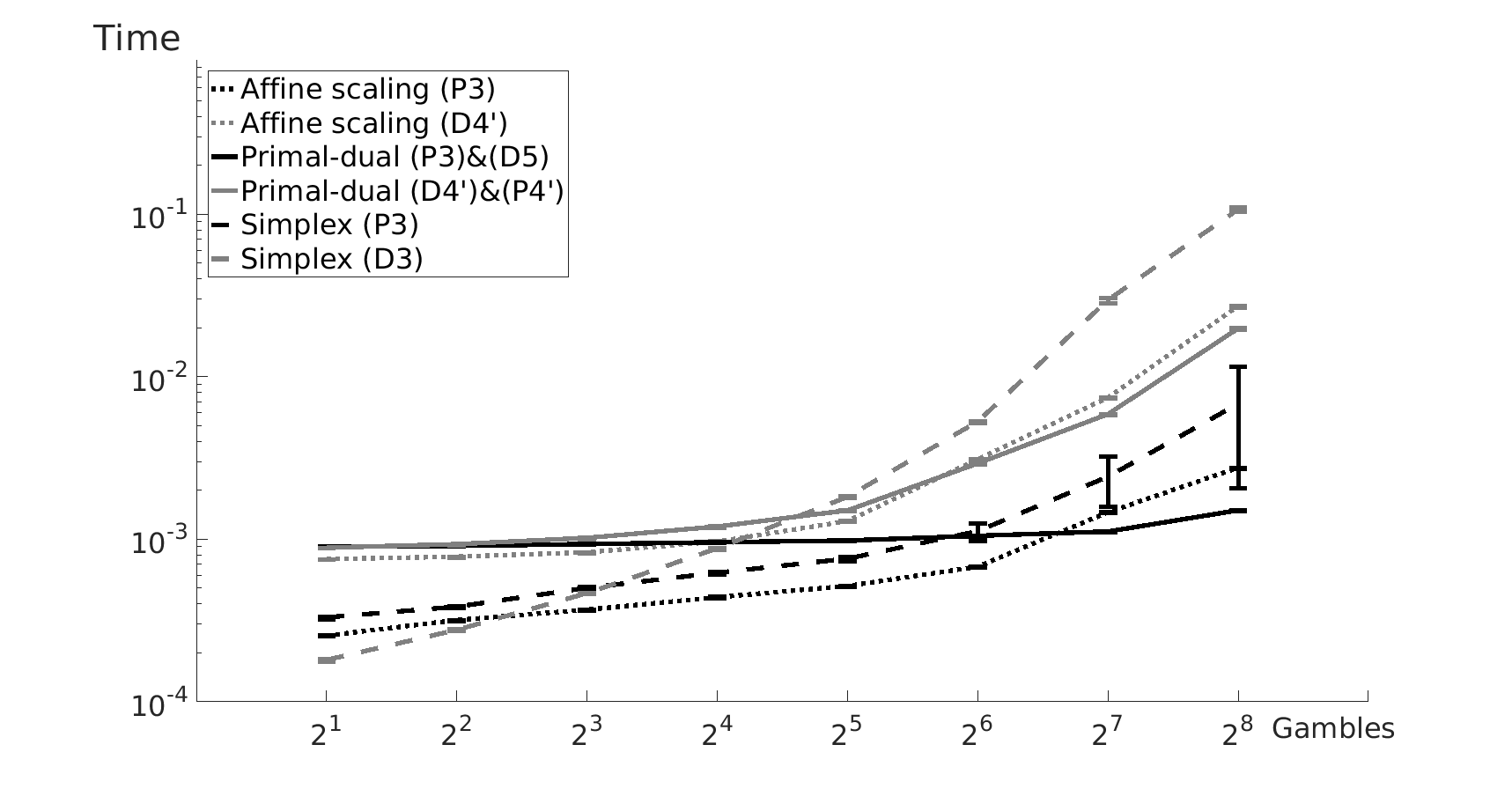}
  \\
  \rotatebox[origin=l]{90}{$|\Omega| = 2^4$}
  &
  \includegraphics[width=\hsize]{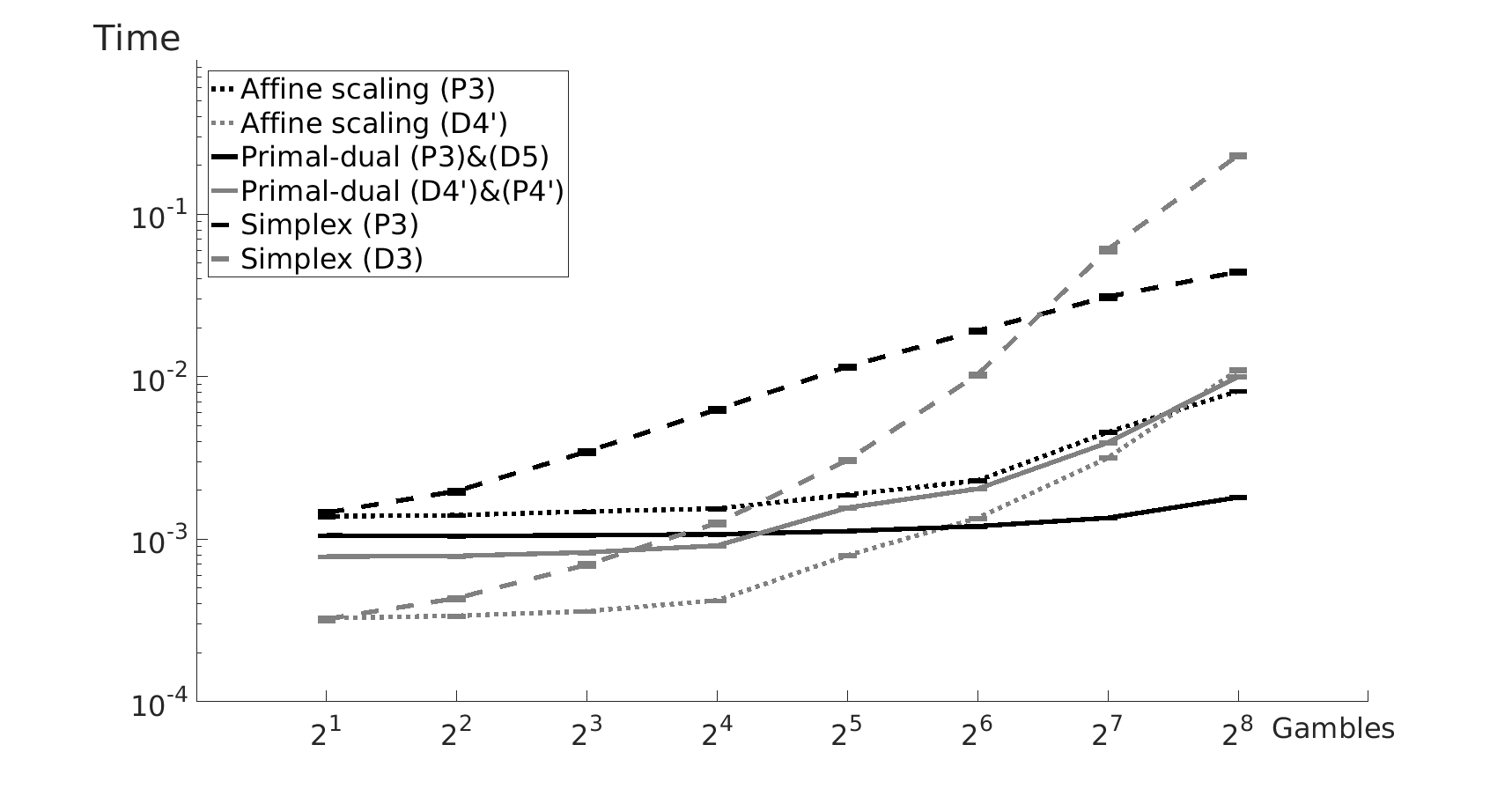}
  &
  \includegraphics[width=\hsize]{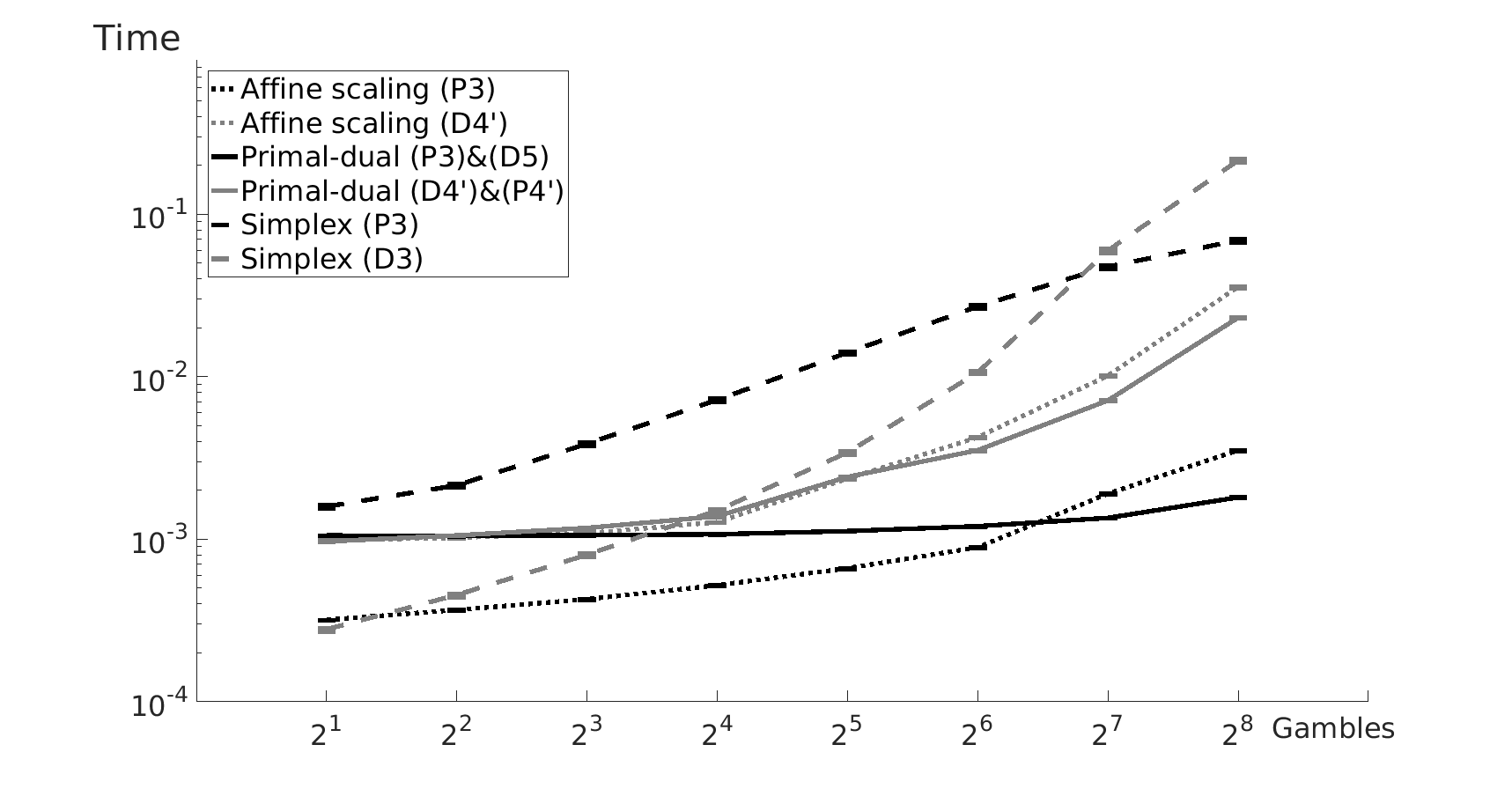}
  \\
  \rotatebox[origin=l]{90}{$|\Omega| = 2^6$}
  &
  \includegraphics[width=\hsize]{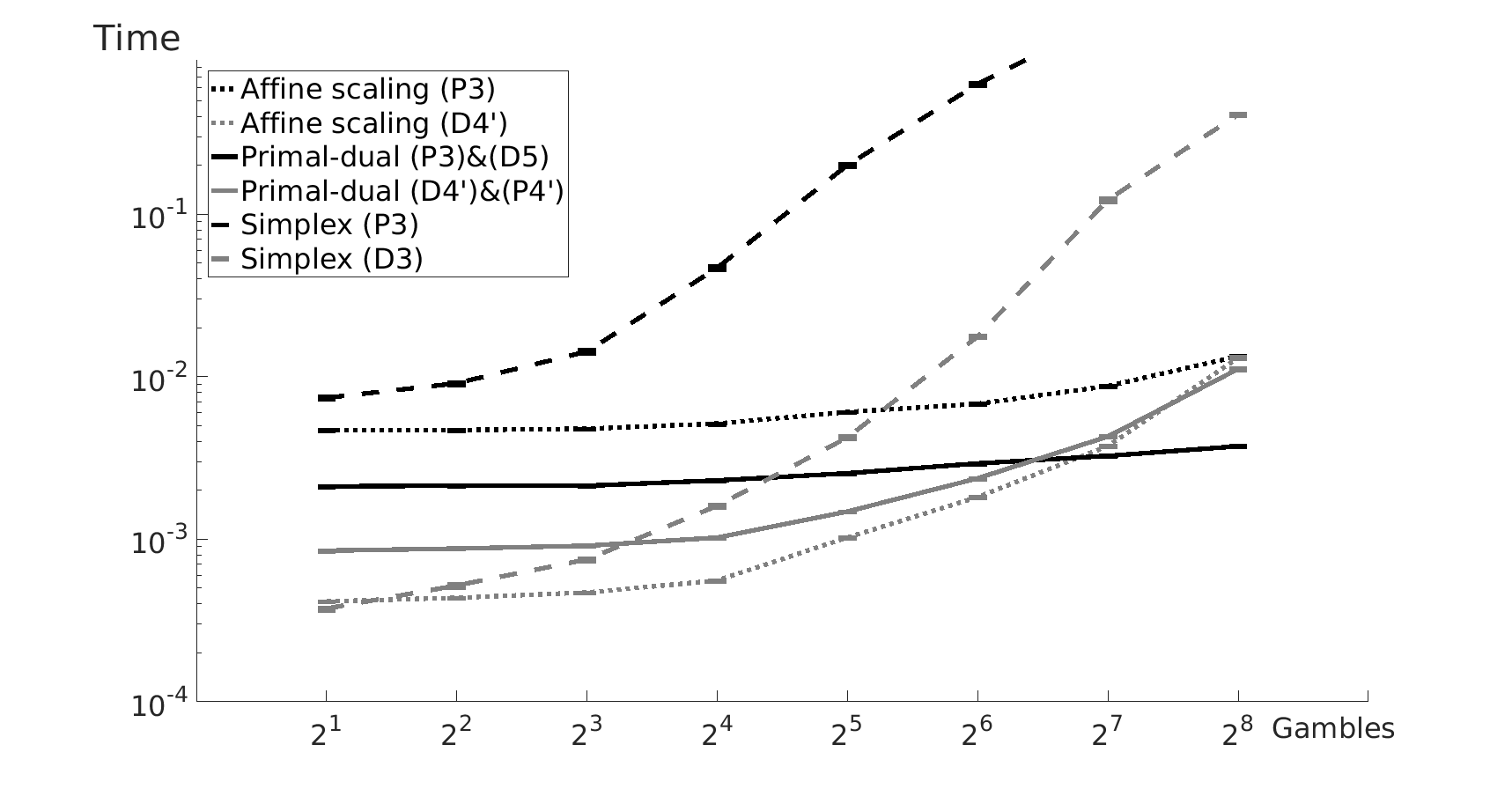}
  &
  \includegraphics[width=\hsize]{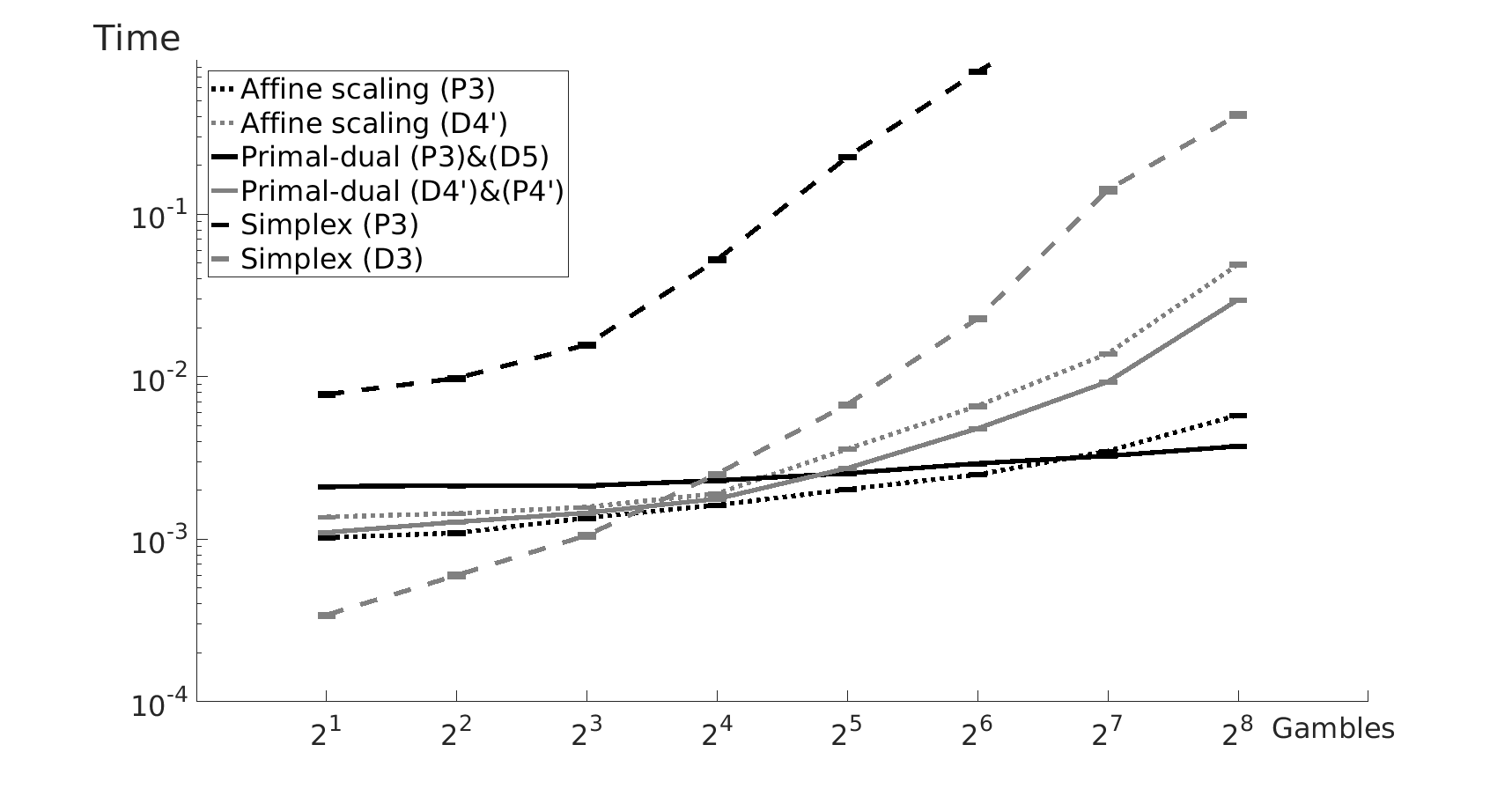}
  \\
  \rotatebox[origin=l]{90}{$|\Omega| = 2^8$}
  &
  \includegraphics[width=\hsize]{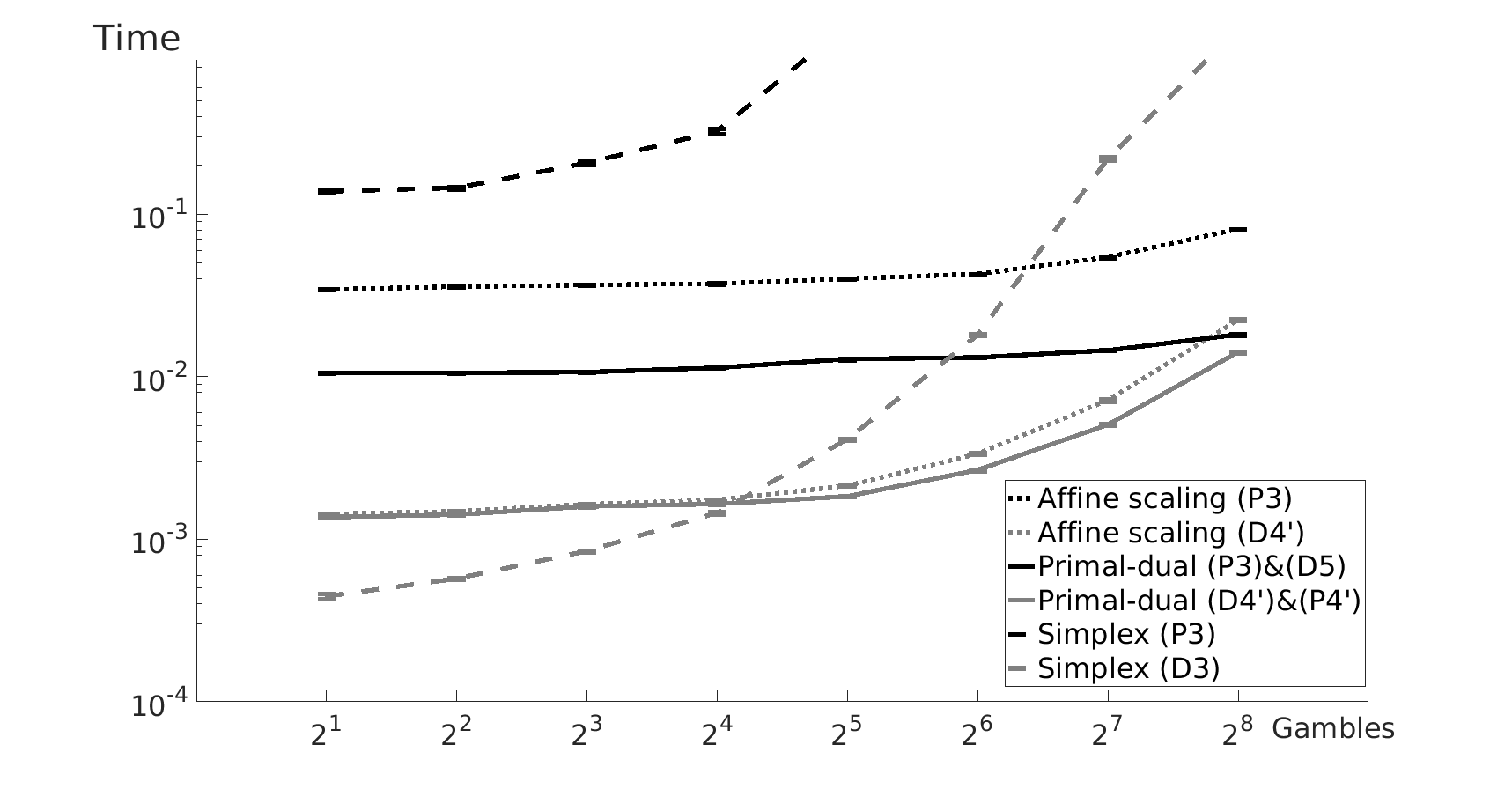}
  &
  \includegraphics[width=\hsize]{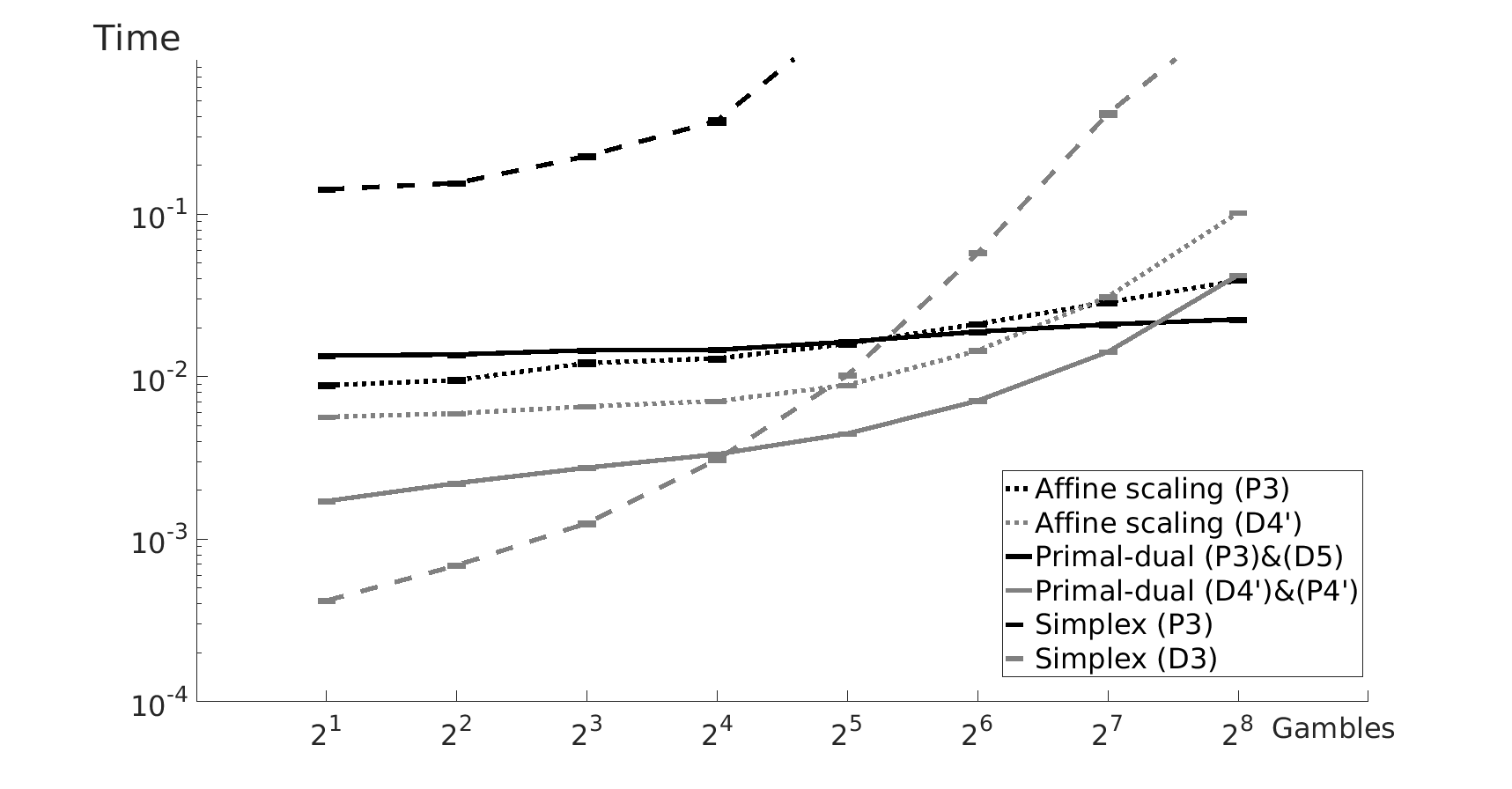}
\end{tabular}
\caption{Comparison plots of the average computational time for three methods. The left column avoids sure loss and the right column does not. Each row represents a different number of outcomes with vary the number of desirable gambles. The labels indicate linear programs solved by different methods.}
\label{fig:7}
\end{figure}

For each method, we run the algorithm twice to remove
any warm-up effects that can happen in the first run, and we only measure the corresponding computational
time taken in the second run. We repeat the process 1000 times and present a summary of the results in \cref{fig:6,fig:7}. 

 \Cref{fig:6,fig:7} show the average computational time taken during each method when checking avoiding sure loss. In the left column, the sets of desirable gambles avoid sure loss while in the right column,
they do not avoid sure loss. In \cref{fig:6} each row represents a different number of desirable gambles, and the horizontal axis represents the number of outcomes. In \cref{fig:7}, each row represents a different number of outcomes, and the horizontal axis represents the number of desirable gambles. In both figures, the vertical axis shows the computational time.
The computational time is averaged over 1000 random sets of desirable gambles. The error bars on the figures represent approximate 95\% confidence intervals on the mean computation time. These
are barely visible because of the large sample size, except in some rare cases where we observed large variability in the simplex method (possibly due to the numerical issues that we discussed earlier in the paper).

In the avoiding sure loss case, note that the sets of desirable gambles are always generated from coherent lower previsions. However, in some applied problems, this may not be the case. Therefore, we ran one further experiment where we introduced a negative bias, which removes coherence whilst still avoiding sure loss. Specifically, in Stage 3 of \cref{alg:ASL}, we set $\mathcal{D}\coloneqq\{f_i-\underline{E}(f_i)+\eta(f_i)\colon i\in\{1,\dots,n\}\}$, for some $\eta(f_i)>0$. 
We considered two scenarios: (i) we uniformly sampled each $\eta(f_i)$ from the open $(0,1)$ interval, and (ii) we fixed each $\eta(f_i)\coloneqq 0.01$. This made no practical difference. In particular, the plots in \cref{fig:6,fig:7} for the avoiding sure loss case remained nearly identical, with no change in general conclusions.

\section{Discussion and conclusion}\label{sec10}

In this study, we discussed and improved the simplex, the affine scaling and the primal-dual methods to efficiently solve linear programming problems for checking whether sets of desirable gambles avoid sure loss or not. To measure our improvements, we also gave several algorithms for generating random sets of desirable gambles that either avoid or do not avoid sure loss.

Building further from \citep{2017:Nakharutai:Troffaes:Caiado}, we studied linear programming problems and two improvements of these methods, namely, (i) an extra stopping criterion, and (ii) a simple and quick algorithm for finding feasible starting points in these methods. We compared the impact of these improvements on the three methods for solving linear programming problems.

These improvements benefit all applied problems as they reduce the computational burden of the original algorithms.
Our benchmarking study quantified these benefits for a wide range of situations.
In case of not avoiding sure loss, we tested the hardest case where only a
single gamble violates consistency. Any
positive computational gain in these cases implies an at least as large gain for more general applied cases where multiple gambles violate consistency.

According to our numerical results, the relative performance of the three methods depends on the number of desirable gambles and the number of outcomes. Specifically, if the number of outcomes is much larger than the number of desirable gambles, then solving either \linprogref{D3} or \linprogref{D4'} is faster than solving \linprogref{P3}. However, if the number of outcomes is much less than the number of desirable gambles, then we prefer to solve \linprogref{P3}. When the two numbers are of similar magnitude, there is no clear difference.

In the results, the primal-dual and the affine scaling methods outperform the simplex method in most cases. The simplex method can find a basic feasible starting point easily, but it cannot apply the extra stopping criterion. Therefore, the simplex method is not a good choice for checking avoiding sure loss.

On the other hand, the affine scaling and the primal-dual methods can benefit from these improvements. Specifically, when we solve \linprogref{P3}, these two methods can apply the extra stopping criterion and a simple mechanism to calculate feasible starting points. In this case, the primal-dual method performs very well, especially when we do not avoid sure loss and the number of desirable gambles is large.

When we solve \linprogref{D4'}, these two interior-point methods can easily find feasible starting points. In this case, the affine scaling method performs very well in small problems whilst the primal-dual method does better when the problems are bigger.  

Overall, if problems are small, then there is no big difference in the time taken to solve either \linprogref{P3} or \linprogref{D4'}, and there is also no big difference between the performance of the methods. When the problems are large, the primal-dual method is the best choice. In this case, if the number of desirable gambles is large, then we solve \linprogref{P3}, and if the number of outcomes is large, then we solve \linprogref{D4'}.

In future work, we will test our improved methods with some real applications and explore an algorithm for choosing $\omega_0$ in \cref{thm:3} for large problems.

\appendix

\section{Proofs}

\begin{proof}[Proof of \cref{thm:3}]
We only show that  $\mathcal{D}$ avoids sure loss if and only if the optimal value of \linprogref{P2} is zero since the proof that the dual problem, \linprogref{D2}, has feasible solutions follows immediately by the strong duality theorem \citep[p.59]{1993:Fang:Puthenpura}.

Firstly, by \cref{lem:1}, the optimal value of \linprogref{P2} is either zero or unbounded. Next, we show that if $\mathcal{D}$ avoids sure loss, then the optimal value of \linprogref{P2} is zero, and vice versa.
Note that \cref{eq:19} can be written as
\begin{equation}\label{eq6.45}
\sup_{\omega\in \Omega} \left( \sum_{i=1}^{n}\lambda_{i}f_{i}(\omega) \right) \leq \sum_{i=1}^{n} \lambda_{i} f_{i}(\omega_0)+\alpha.
\end{equation}  
Suppose $\mathcal{D}$ avoids sure loss, then by \cref{def:47}, for all $n$, all $\lambda_{1}$, \dots, $\lambda_{n} \geq 0$, and $f_{1}$, \dots, $f_{n} \in \mathcal{D}$, 
\begin{equation}\label{eq6.46}
 0 \leq \sup_{\omega \in \Omega} \left( \sum_{i=1}^{n} \lambda_{i}f_{i}(\omega) \right). 
\end{equation} 
So, by \cref{eq6.45},
\begin{equation}\label{eq6.47}
 0 \leq \sum_{i=1}^{n} \lambda_{i} f_{i}(\omega_0)+\alpha. 
\end{equation} 
So, the optimal value is non-negative.
Now, by putting $\lambda_{i} =0$ for all $i$, and $\alpha =0$, we obtain
\begin{equation}\label{eq6.48}
 \sum_{i=1}^{n} \lambda_{i} f_{i}(\omega_0)+\alpha = 0. 
\end{equation}
Therefore, the optimal value of \linprogref{P2} is zero.

Conversely, suppose $\mathcal{D}$ does not avoid sure loss. There are non-negative $\lambda_{1}$, \dots, $\lambda_{n}$ such that
\begin{equation}\label{eq6.49}
\sup_{\omega \in \Omega} \left( \sum_{i=1}^{n} \lambda_{i}f_{i}(\omega) \right) <  0 . 
\end{equation} 
Set \begin{equation}
s = \sup_{\omega \in \Omega} \left( \sum_{i=1}^{n} \lambda_{i}f_{i}(\omega) \right) 
\end{equation} 
and choose
\begin{equation}
\alpha = s - \sum_{i=1}^{n} \lambda_{i} f_{i}(\omega_0).
\end{equation}
Now we have $\alpha \geq 0$ and 
\begin{equation}\label{eq6.50}
\forall \omega \neq \omega_0\colon \sum_{i=1}^{n}\lambda_{i}f_{i}(\omega)  \leq \sum_{i=1}^{n} \lambda_{i} f_{i}(\omega_0)+\alpha = s < 0.
\end{equation}
This means that 
\begin{equation}\label{eq6.51}
\sum_{i=1}^{n} \lambda_{i} f_{i}(\omega_0)+\alpha < 0 
\end{equation}
is a feasible value of \linprogref{P2}. By \cref{lem:1}, the optimal value is unbounded.
\end{proof}

\section*{Acknowledgements}

We would like to acknowledge support for this project from the Development and Promotion of Science and Technology Talents Project (Royal Government of Thailand scholarship).

\bibliographystyle{plainnat}
\bibliography{references}

\end{document}